\documentclass[11pt]{article}
\usepackage{listings}
\usepackage{pdflscape}
\usepackage{amsfonts}
\usepackage{epsfig}
\usepackage{lscape}
\usepackage{amsmath,amssymb,amsthm,mathtools,cancel,mathrsfs}
\usepackage{graphicx,subcaption}
\usepackage{color}
\usepackage{placeins}
\usepackage{url}
\usepackage{cases,empheq}
\usepackage{longtable}
\usepackage{supertabular}

\allowdisplaybreaks[4]
\usepackage{cite}
\usepackage{algorithm,algorithmic,setspace,float}
\usepackage{booktabs,multirow,array,threeparttable}
\usepackage{enumitem}
\usepackage{lipsum}
\usepackage{hyperref}
\usepackage{cleveref}
\usepackage{svg}



\newtheorem{theorem}{Theorem}
\newtheorem{example}{Example}
\newtheorem{assumption}{Assumption}
\newtheorem{lemma}{Lemma}
\newtheorem{remark}{Remark}
\newtheorem{definition}{Definition}
\newtheorem{proposition}{Proposition}

\newcommand{\mc}{\mathcal}
\newcommand{\mb}{\mathbb}
\newcommand{\mr}{\mathrm}

\newcommand{\sgn}{\operatorname{sgn}}
\newcommand{\prox}{\operatorname{prox}}

\newcommand{\conv}{\operatorname{conv}}
\newcommand{\dom}{\operatorname{dom}}
\newcommand{\supp}{\operatorname{supp}}

\newcommand{\argmin}{\operatornamewithlimits{arg\,min}}
\newcommand{\argmax}{\operatornamewithlimits{arg\,max}}

\usepackage{caption}
\usepackage{rotating}
\usepackage{epstopdf}
\newenvironment{keywords}{\par\vspace{0.5\baselineskip}\noindent\textbf{Keywords:} }{\par}

\begin{document}
\title{\bf A Perturbed DCA for Computing d-Stationary Points of Nonsmooth DC Programs\footnotemark[1]}
\author{Zhangcheng Feng\footnotemark[2] \; and 
Yancheng Yuan\footnotemark[3]}
\date{31 July, 2026}
\maketitle

\renewcommand{\thefootnote}{\fnsymbol{footnote}}
\footnotetext[1]{{\bf Funding:} The work of Yancheng Yuan was supported by the RGC Early Career Scheme (Project No.
25305424) and the Research Center for Intelligent Operations Research.}
\footnotetext[2]{Department of Applied Mathematics, The Hong Kong Polytechnic University, Hung Hom, Hong Kong ({\tt zhangcheng.feng@connect.polyu.hk}).}
\footnotetext[3]{Department of Applied Mathematics, The Hong Kong Polytechnic University, Hung Hom, Hong Kong (\textbf{Corresponding author}. {\tt yancheng.yuan@polyu.edu.hk}).}
\renewcommand{\thefootnote}{\arabic{footnote}}
\begin{abstract}
This paper introduces an efficient perturbed difference-of-convex algorithm (perturbed DCA) for computing d-stationary points of an important class of structured nonsmooth difference-of-convex problems. Compared to the principal algorithms introduced in [J.-S. Pang, M. Razaviyayn, and A. Alvarado, Math. Oper. Res. 42(1):95--118 (2017)], which may require solving several subproblems for a one-step update, perturbed DCA only requires solving a single subproblem. Therefore, the per-iteration computational cost of perturbed DCA is comparable to the widely used difference-of-convex algorithm (DCA) introduced in [D. T. Pham and H. A. Le Thi, Acta Math. Vietnam. 22(1):289--355 (1997)] for computing a critical point. {We establish the subsequential and almost sure convergence of the perturbed DCA to d-stationary points under certain conditions. To decouple the perturbation radii from the local convergence rate of the iterates, we further propose a hybrid variant of the perturbed DCA that independently samples the perturbation radius and direction with a safeguard using a proximal DCA step.}
Importantly, under more relaxed and practical assumptions, we prove that every accumulation point of the sequence generated by the {hybrid} perturbed DCA is a d-stationary point almost surely. Numerical results on several important examples demonstrate the efficiency of the proposed methods for computing d-stationary points.
\end{abstract}


\section{Introduction}
This paper considers a class of difference-of-convex (DC) programming problems of the following form:
\begin{equation}\label{eq:dc_general}
    \min_{x \in \mathbb{R}^n}~\zeta(x) := \phi(x) - \psi(x), \quad \psi(x) := \max_{i \in \mathcal{I}}~\{\psi_i(x)\},
\end{equation}
where $\mathcal{I}$ is a finite index set, $\phi: \mathbb{R}^n \to (-\infty, +\infty]$ is a proper closed convex function, and for each $i \in \mathcal{I}$, $\psi_i: \mathbb{R}^n \to \mathbb{R}$ is convex and continuously differentiable. We assume that the functions $\{\psi_i\}_{i\in\mc I}$ are pairwise distinct. This formulation provides a unified framework for modeling and solving a broad class of nonconvex optimization problems that arise in various applications, including sparse recovery \cite{ahn2017dclearning,gotoh2018dcsparse}, statistical estimation \cite{cui2018composite}, and power allocation in digital communication systems \cite{pang2017computing}. We refer readers to \cite{tao1997convex,horst1999dcoverview,le2018dc30years,le2024open, cui2021modern,facchinei2007finite} and references therein for more details on DC programming.

\begin{algorithm}[htbp]
\caption{Difference-of-Convex Algorithm (DCA) \cite{tao1997convex,an2005dc}}
\label{alg:DCA}
\begin{algorithmic}[1]
\setlength{\abovedisplayskip}{2pt}
\setlength{\belowdisplayskip}{2pt}
\REQUIRE $x^0 \in \mathbb{R}^n$.
\FOR{$k=0,1,\cdots$}
    \STATE \textbf{Step 1}. Select a subgradient $g^k \in \partial \psi(x^k)$.
    \STATE \textbf{Step 2}. Update
    \begin{equation}
    \label{eq: dca_update}
    x^{k+1} \in \argmin_{x \in \mathbb{R}^n} \left\{ \phi(x) - \psi(x^k) - \langle g^k, x - x^k \rangle \right\}.
    \end{equation}
\ENDFOR
\end{algorithmic}
\end{algorithm}

The standard difference-of-convex algorithm (DCA) \cite{tao1997convex,an2005dc} shown in Algorithm~\ref{alg:DCA} is one of the most widely used approaches for solving DC problems due to its simplicity and ease of implementation. At each iteration $k \ge 0$, DCA linearizes the concave part of the objective and computes the next iterate by solving the convex subproblem \eqref{eq: dca_update}. Simpler subproblems can be obtained by using a specific DC decomposition described in \cite{tao1998dctrustregion}. Specifically, one can always strongly convexify the functions $\phi$ and $\psi$ by adding the same strongly convex functions to both components without changing the objective function $\zeta$ in \eqref{eq:dc_general}. This idea has been adopted in \cite{gotoh2018dcsparse} to propose the so-called proximal DCA, which solves the following strongly convex subproblem for updating $x^{k+1}$ at iteration $k$:
$$ x^{k+1} = \argmin_{x\in\mb R^n} \left\{\phi(x) - \psi(x^k) - \langle g^k, x-x^k \rangle + \frac{\sigma}{2}\|x - x^k\|^2 \right\}, $$
where $\sigma>0$ is given. Wen et al. \cite{wen2018pdcae} further incorporated an extrapolation step into the proximal DCA to achieve possible acceleration. The resulting algorithm exhibits better numerical performance in some applications.

While DCA and its variants are popular and demonstrate good practical performance in many applications, they only guarantee to converge to some \textit{critical points} in general. To address this issue, Pang et al.~\cite{pang2017computing} advocate for the stronger concept of \emph{directional-stationary (d-stationary) points}, where the directional derivative is nonnegative in every feasible direction, and is arguably the sharpest kind among the various stationary solutions for the DC programs. One can refer to \cite{pang2017computing, cui2021modern} and the references therein for a more detailed discussion. Moreover, they proposed a novel revised DCA \cite[Alg.~1]{pang2017computing} shown in Algorithm~\ref{alg:Pang} to compute d-stationary points of the nonsmooth DC programs \eqref{eq:dc_general} by developing the \textit{relaxed active index set}. In particular, at each iteration $k$, instead of solving a single subproblem in DCA, the revised algorithm will solve all strongly convex linearized subproblems with proximal term~\eqref{eq: subproblem_pang} identified by an $\epsilon$-active index set~\eqref{eq: epsilon-activeset}, and update $x^{k+1}$ by selecting the best one with smallest ``proximal objective value''. Such an $\epsilon$-active index set can identify all active pieces at the cluster point, which can contribute to guaranteeing a subsequential convergence to a d-stationary point of the nonsmooth DC programming problem~\eqref{eq:dc_general}. However, the computational cost for a one-step update could be expensive or even prohibitive for large-scale problems. On the one hand, identifying the $\epsilon$-active index set is not easy in general for large-scale problems. On the other hand, when the $\epsilon$-active index set $\mathcal{M}_\epsilon(x^k)$ is large, it requires solving multiple large-scale subproblems. To mitigate the computational challenge, randomized active-index methods were developed in \cite[Sec.~5.2]{pang2017computing} {and \cite[Alg.~2]{van2019non}}, which randomly select a single index from $\mathcal{M}_\epsilon(x^k)$ per iteration and solve the corresponding subproblem.
Moreover, a safeguard strategy of \textit{sufficient-decrease condition} is included to ensure the almost sure convergence to a d-stationary point of~\eqref{eq:dc_general}. When $\mc M_\epsilon(x^k)$ is large but only a few indices yield candidates satisfying this condition, a random draw has a low probability of selecting such an index. The resulting inefficient sampling can lead to many rejected candidates and substantially increase the total number of subproblems to be solved.

\begin{algorithm}[htbp]
\caption{A revised DCA for computing d-stationary points of \eqref{eq:dc_general}~\cite[Alg.~1]{pang2017computing}}
\label{alg:Pang}
\begin{algorithmic}[1]
\setlength{\abovedisplayskip}{2pt}
\setlength{\belowdisplayskip}{2pt}
\REQUIRE $x^0 \in \mathbb{R}^n$, $\epsilon > 0$.
\FOR{$k=0,1,\cdots$}
    \STATE Identify the $\epsilon$-active index set:
    \begin{equation}
    \label{eq: epsilon-activeset}
    \mathcal{M}_\epsilon(x^k) := \{ i \in \mathcal{I} \mid \psi_i(x^k) \ge \psi(x^k) - \epsilon \}.
    \end{equation}
    \FOR{$i \in \mathcal{M}_\epsilon(x^k)$}
        \STATE Compute
        \begin{equation}
        \label{eq: subproblem_pang}
        \hat{x}^{k,i} = \argmin_{x \in \mathbb{R}^n} \left\{ \phi(x) - \psi_i(x^k) - \langle \nabla \psi_i(x^k), x - x^k \rangle + \frac{1}{2} \|x - x^k\|^2 \right\}.
        \end{equation}
    \ENDFOR
    \STATE Find
    \begin{equation*}
    \hat{i} \in \argmin_{i \in \mathcal{M}_\epsilon(x^k)} \left\{ \zeta(\hat{x}^{k,i}) + \frac{1}{2} \|\hat{x}^{k,i} - x^k\|^2 \right\}.
    \end{equation*}
    \STATE Update $x^{k+1} = \hat{x}^{k,\hat{i}}$.
\ENDFOR
\end{algorithmic}
\end{algorithm}

While numerous enhanced versions of Algorithm \ref{alg:Pang} have been proposed and successfully computed d-stationary points across various scenarios (e.g., \cite{pang2018decomposition,zhang2026data,cui2018composite,sun2024hybrid,lu2019enhanced}), they indeed depend on the relaxed active index set technique, which remains computationally expensive for solving large-scale problems. {In addition, the relaxed active index set technique can be combined  with proximal stabilization or bundle ideas, such as the stability-center mechanism in \cite[Alg.~2]{de2019proximal}. Such proximal/stabilized variants can keep the $\epsilon$-active index set fixed until a serious step is accepted, and hence may reduce the frequency of recomputing relaxed active index sets. Nevertheless, these variants still require constructing or maintaining the corresponding $\epsilon$-active index sets and typically test a sufficient-decrease condition at each iteration.} Therefore, designing a new algorithm that can balance the tradeoff between the theoretical sharpness and the computational efficiency for solving large-scale nonsmooth DC problem~\eqref{eq:dc_general} is highly desirable.

In this paper, we address this challenge by first proposing a (randomly) perturbed difference-of-convex algorithm \textbf{(perturbed DCA)}. The key idea of perturbed DCA is to linearize the convex component $\psi(\cdot)$ at a randomly perturbed point $\hat{x}^k$, rather than at the current iterate $x^k$. By Rademacher’s theorem~\cite{Rademacher1919,rw1998}, the active-gradient set at $\hat{x}^k$ is a singleton almost surely. 
This allows perturbed DCA to only solve a single strongly convex subproblem at each iteration. {Unlike randomized (relaxed) active-index methods (e.g., \cite[Sec.~5.2]{pang2017computing}, \cite[Alg.~2]{van2019non}), perturbed DCA does not first construct an $\epsilon$-active index set and then sample an index from it. The random perturbation implicitly selects a unique active gradient at a nearby differentiability point almost surely. 
We show that, under certain conditions, the perturbed DCA has subsequential and almost sure convergence to a d-stationary point. In particular, it requires a relative-rate condition: along a suitable subsequence converging to a cluster point \(x^*\), the distance \(\|x^k-x^*\|\) must be bounded by a constant multiple of the perturbation radius \(\alpha_k\). This condition ensures that the random perturbation can reach each relevant active region near \(x^*\) with positive probability, but it couples the choice of \(\alpha_k\) with the unknown local behavior of the iterates, thus is challenging to verify in practice.
To bypass this requirement, we further propose a hybrid variant of the perturbed DCA. This algorithm simultaneously randomizes the perturbation radius and direction, and incorporates a fallback proximal DCA step to ensure sufficient decrease. Consequently, it decouples the perturbation radii from the local convergence rate of the iterates while retaining the almost-sure d-stationarity guarantee.}

{The random perturbation mechanism in this paper is related in spirit to gradient sampling methods~\cite{burke2005gs,davis2022gs}. However, our methods can obtain stronger stationary points with much less computational cost. First, gradient sampling methods typically require multiple randomly perturbed differentiable points nearby (larger than the dimension of the problem) to approximate the Clarke subdifferential and then solve an auxiliary subproblem to compute a search direction. In contrast, our perturbed DCA only requires one perturbed point at which the active-gradient set of $\psi(\cdot)$ is a singleton, and then performs the update at this perturbed point.
Second, under practical assumptions, our methods converge subsequentially to a d-stationary point almost surely, whereas gradient sampling methods only converge to Clarke stationary points almost surely~\cite{burke2005gs,davis2022gs}. It is known that Clarke stationarity is weaker than d-stationarity \cite[Sec.~3.2]{pang2017computing}.
}

We summarize the main contributions of this paper as follows:
\begin{itemize}
    \item We design a simple yet effective algorithm called perturbed DCA for solving the nonsmooth DC program~\eqref{eq:dc_general}.
    The proposed perturbed DCA balances the theoretical sharpness and computational efficiency. In particular, perturbed DCA only solves a single strongly convex subproblem for a one-step update, and it guarantees subsequential convergence to a d-stationary point almost surely {under certain conditions}.
    \item {To decouple the perturbation radii from the local convergence rate of the iterates, we propose a hybrid variant of perturbed DCA with randomized perturbation radii and a fallback proximal DCA step. The randomized radii enable the method to sample every relevant active region at arbitrarily fine scales, while the sufficient-decrease safeguard preserves monotonic descent. We prove that every accumulation point is d-stationary almost surely under more relaxed and practical
assumptions}.
    \item Extensive numerical experiment results on sparse regression and clustering problems demonstrate the efficiency of {the proposed perturbed DCA and its hybrid variant} for computing d-stationary points of \eqref{eq:dc_general}.
\end{itemize}

The remainder of this paper is organized as follows. Section~\ref{sec:DC} introduces some necessary preliminaries and motivating examples. Section~\ref{sec:pDCA} introduces the proposed perturbed DCA and establishes its convergence guarantees. {Section~\ref{sec:hybrid-pDCA} presents the hybrid variant of perturbed DCA.} Numerical experiment results are presented in Section~\ref{sec:numerical}, and we conclude the paper in Section~\ref{sec:conclusion}.

\paragraph{Notations} Let $\mb N$ denote the set of natural numbers. Given an integer $n\geq1$, let $[n]:=\{1,\ldots,n\}$. Denote by $\mb R^n$ the $n$-dimensional real Euclidean space, and $\mb S^{n-1}:=\{x\in\mb R^n \mid \|x\|=1 \}$ the unit sphere. Let $\conv(X)$ denote the convex hull of the set $X$. Let $x\in\mb R^n$ be any given vector. Denote $\sgn(x)=[\sgn(x_1),\ldots,\sgn(x_n)]^\top$ the sign vector of $x$ and $\supp(x)=\{i\in[n]\mid x_i\neq0\}$ the support set of $x$. Denote the $\ell_0$-norm of $x$ by $\|x\|_0$, which is the number of non-zero elements in $x$. For any $p \geq 1$, denote the $\ell_p$-norm of $x$ as $\|x\|_p := \sqrt[p]{\sum_{i = 1}^n |x_i|^p}$. For any integer $1 \leq k \leq n$, denote the vector $k$-norm of $x$ as $\|x\|_{(k)} := \sum_{1 \leq i \leq k} |x|^{\downarrow}_i$, where $|x|^{\downarrow}$ is the vector obtained by
rearranging the absolute value of the coordinates of $x$ in non-increasing order. For any convex function $f:\mb R^n\to(-\infty,\,+\infty]$, denote its domain by $\dom(f)$, its subdifferential at $x\in\mb R^n$ as $\partial f(x):=\{g\in\mb R^n\mid f(y)-f(x)\geq\langle g, y-x \rangle,\,\forall y\in\mb R^n\}$, and its proximal mapping as $\prox_{\tau f}(x):=\argmin_{y\in\mb R^n}\{f(y)+\frac{1}{2\tau}\|x-y\|^2\}$, $\tau>0$. If $f$ is differentiable at $x$, $\partial f(x)=\{\nabla f(x)\}$, where $\nabla f(x)$ is the gradient of $f$ at $x$.

\section{Preliminaries and Motivating Examples}\label{sec:DC}

The DC program~\eqref{eq:dc_general} provides a unified framework to model a wide class of nonsmooth and nonconvex optimization problems. We start with some interesting motivating examples to partially illustrate the applicability of the formulation \eqref{eq:dc_general}.

\begin{example}[$K$-sparse regularized problems~{\cite{gotoh2018dcsparse,ahn2017dclearning}}]\label{ex:ksparsity}
Sparse learning with cardinality constraint has been widely explored, while finding an optimal solution in general is known to be NP-hard due to the combinatorial nature of the cardinality constraints. Using the fact that $\|x\|_0 \leq K$ is equivalent to $\|x\|_1 - \|x\|_{(K)} = 0$ for any $x \in \mathbb{R}^n, \; 1 \leq K \leq n$, the following $K$-sparse regularized problem has been extensively investigated recently:
\begin{equation}\label{eq:dc_ksparse}
    \min_{x\in\mb R^n}\quad f(Ax) + \lambda(\|x\|_1-\|x\|_{(K)}),
\end{equation}
where $A\in\mb R^{m\times n}$, $f: \mathbb{R}^m \to \mathbb{R}$ is a closed convex loss function, and $\lambda>0$ is a given parameter. In particular, when taking $f(y) = \frac{1}{2}\|y - b\|^2 \; \forall y \in \mathbb{R}^m$ for some given $b \in \mathbb{R}^m$, it becomes the $K$-sparse regularized linear regression problem
\begin{equation}\label{eq:dc_ksparse-lr}
    \min_{x\in\mb R^n}\quad \frac12\|Ax-b\|^2 + \lambda(\|x\|_1-\|x\|_{(K)}),
\end{equation}
which has numerous applications in signal processing and compressed sensing.  The relationship between the $K$-sparse regularized problem \eqref{eq:dc_ksparse} and its corresponding cardinality constrained problem has been extensively investigated in \cite{ahn2017dclearning}.

Note that $||x||_{(K)}$ can be expressed as the pointwise maximum of finitely many linear functions as follows:
\begin{equation*}
    ||x||_{(K)}=\max\left\{ \langle \nu, x\rangle=\sum_{i=1}^n\nu_ix_i \mid \nu \in\{-1, 0, 1\}^n, \; |\supp(\nu)| = K  \right\}.
\end{equation*}
Therefore, the problem~\eqref{eq:dc_ksparse} is in the form of the optimization problem~\eqref{eq:dc_general}.
\end{example}

Regularization is one of the key techniques for feature selection in high dimensional statistical learning. Most of the popular nonconvex regularizers are DC functions, including the smoothly clipped absolute deviation (SCAD) penalty \cite{fan2001variable}, the minimax concave penalty \cite{zhang2010nearly}, and the capped $\ell_1$ penalty \cite{zhang2010cappedl1}. We describe the DC formulation of the capped $\ell_1$ penalty below. The DC formulation of more popular nonconvex regularizers can be found in the recent monograph \cite{cui2021modern}.

\begin{example}[Capped $\ell_1$ penalty~\cite{zhang2010cappedl1}]
\label{ex:capped_l1}
Consider the problem
\begin{align*}
\min_{x \in \mathbb{R}^n}\quad f(Ax) + \lambda \sum_{i=1}^n \min\left\{|x_i|, \theta\right\},
\end{align*}
where $A\in\mathbb{R}^{m\times n}$, $f: \mathbb{R}^m \to \mathbb{R}$ is a closed convex loss function, and $\lambda > 0$ and $\theta > 0$ are given parameters. The capped $\ell_1$ penalty term can be equivalently expressed as a DC function:
\begin{align*}
\sum_{i=1}^n \min\left\{|x_i|, \theta\right\} = \|x\|_1 - \sum_{i=1}^n \max\left\{|x_i| - \theta, 0\right\}.
\end{align*}
Thus, the problem can be rewritten in the form of \eqref{eq:dc_general} with:
\begin{align*}
\phi(x) = f(Ax) + \lambda \|x\|_1, \quad \psi(x) = \lambda \sum_{i=1}^n \max\left\{x_i - \theta,\, -x_i-\theta,\, 0\right\}.
\end{align*}
Furthermore, the function $\psi(x)$ admits a pointwise maximum representation:
\begin{align*}
    \psi(x)=\max\left\{ \lambda\sum_{i\in\supp(s)}(s_ix_i-\theta)\mid s \in \{-1,0,1\}^n \right\}.
\end{align*}
\end{example}

\begin{example}[DC reformulations of $K$-means and $K$-medians clustering models]
\label{ex:ex-clustering}
Given a dataset $\{a_1, \dots, a_n\} \subseteq \mathbb{R}^d$ and the number of clusters $1 < K < n$, we consider the following nonconvex nonsmooth clustering problem:
\begin{equation}
\label{eq:k-cluster}
\min_{\mu = (\mu_1, \dots, \mu_K) \in \mathbb{R}^{d \times K}} ~ \frac{1}{n}\sum_{i=1}^n \min_{1 \leq j \leq K} \|\mu_j - a_i\|_p^p,
\end{equation}
where $p = 1, 2$. The clustering model \eqref{eq:k-cluster} is known as the $K$-medians model and the $K$-means model for $p = 1$ and $p = 2$, respectively. By realizing that 
\begin{align*}
\min_{1 \leq j \leq K} \|\mu_j - a_i\|_p^p = \sum_{l=1}^K \|\mu_l - a_i\|_p^p - \max_{1 \leq j \leq K} \sum_{1 \leq l \neq j \leq K}\|\mu_l - a_i\|_p^p, \quad \forall 1 \leq i \leq n,
\end{align*}
we can equivalently reformulate the clustering model \eqref{eq:k-cluster} as the following DC programming:
\begin{equation}
\label{eq:dc_clustering}
\min_{\mu \in \mathbb{R}^{d \times K}}~\underbrace{\frac{1}{n}\sum_{i=1}^n \sum_{l=1}^K \|\mu_l - a_i\|_p^p}_{\phi(\mu)} - \underbrace{\frac{1}{n}\sum_{i = 1}^n \max_{1 \leq j \leq K} \left(\sum_{1 \leq l \neq j \leq K} \|\mu_l - a_i\|_p^p\right)}_{\psi(\mu)}.
\end{equation}
Moreover, let $\mathcal{I}=[K]^n$ denote the $n$-fold Cartesian product of $[K]$, i.e., $\mathcal{I}:=\{(j_1,\ldots,j_n)\mid j_i\in \{1, 2, \dots, K\},\,i=1,\ldots,n\}$, and define 
\begin{align*}
\psi_j(\mu) := \frac{1}{n}\sum_{i=1}^n\left(\sum_{1\leq l \neq j_i\leq K}\|\mu_l-a_i\|_p^p\right) \quad \forall j = (j_1, \dots, j_n) \in \mathcal{I},
\end{align*}
we know
\begin{align*}
\psi(\mu) = \max_{j \in \mathcal{I}}~\{\psi_j(\mu)\}.
\end{align*}
Therefore, both $K$-means and $K$-medians clustering problems can be written in the form of the optimization problem \eqref{eq:dc_general}. It has been demonstrated that the DC reformulations of $K$-means and $K$-medians models can achieve better performance \cite{an2007new}.
\end{example}

Next, we move on to introduce some definitions and assumptions that are necessary for the rest of this paper. Recall that the problem~\eqref{eq:dc_general} has the following form:
\begin{equation*}
    \min_{x\in\mb R^n}\quad \zeta(x)=\phi(x)-\psi(x),\quad \psi(x)=\max_{i\in \mc I}~\{\psi_i(x)\}.
\end{equation*}
We define  the \textit{active index set} of $\psi$ at $x$ as
\begin{align}\label{eq:active_set_general}
    \mc M(x) := \argmax_{i\in \mc I}\{\psi_i(x)\}=\left\{ i\in \mc I \mid \psi_i(x)=\psi(x) \right\}.
\end{align}
Since ${\rm dom}(\psi_i) = \mathbb{R}^n$ for all $i \in \mathcal{I}$ and $\phi(\cdot)$ is a proper closed convex function, we know that $\zeta(\cdot)$ is directionally differentiable on ${\rm dom}(\zeta)$ \cite[Sec.~23]{rockafellar1970convex}.  The directional derivative of $\zeta(\cdot)$ at $x \in {\rm dom}(\zeta)$ along a direction $d \in \mathbb{R}^n$ is defined as
\begin{align*}
\zeta'(x; d) := \lim_{\tau\downarrow0}\frac{\zeta(x +\tau d)-\zeta(x)}{\tau}.
\end{align*}

The optimization problem \eqref{eq:dc_general} is nonconvex, and it is not possible to compute a global solution in general. In practice, we will focus on computing a good stationary solution to \eqref{eq:dc_general} instead. The concepts of critical points and d-stationary points are defined below. 

\begin{definition}[Critical point~\cite{tao1997convex}]
    We call $x^*\in {\rm dom}(\zeta)$ a critical point of the optimization problem \eqref{eq:dc_general} if 
    $$\partial \phi(x^*) \bigcap \partial\psi(x^*) \neq \emptyset.$$
\end{definition}

\begin{definition}[d-stationary point~\cite{pang2017computing}]
    We call $x^*\in {\rm dom}(\zeta)$ a \textit{d-stationary} point of the optimization problem \eqref{eq:dc_general} if the directional derivative 
    $$ \zeta'(x^*;x-x^*) \geq 0, \quad \forall x\in\mb R^n.$$
\end{definition}
It is known that d-stationary points are stronger than the critical points in general \cite{pang2017computing}. A detailed discussion of the relationships among different stationary points for the DC programs can be found in \cite{pang2017computing}. The following characterization of the d-stationary points of \eqref{eq:dc_general} is useful. 

\begin{proposition}[{\cite[Prop.~5]{pang2017computing}}]\label{pro:d-stationary}
    A vector $x^*\in\dom(\zeta)$ is a d-stationary point of the problem~\eqref{eq:dc_general} if and only if: $\forall i\in\mc M(x^*)$,
    $$ x^*=\argmin_{x\in\mb R^n}\left\{ \phi(x)-\langle \nabla\psi_i(x^*),x-x^* \rangle+\frac{\sigma}{2}\|x-x^*\|^2 \right\}, $$ where $\sigma>0$. Moreover, if the function $\psi$ is piecewise affine, then any such d-stationary point is a local minimizer of $\zeta$.
\end{proposition}

The following two assumptions regarding the objective functions in \eqref{eq:dc_general} are assumed for the rest of this paper.

\begin{assumption}[Bounded level sets]\label{ass:coercive}
    The level sets of $\zeta$ are bounded, i.e., for any $c \in \mathbb{R}$, the set $\{x \mid \zeta(x) \leq c\}$ is bounded.
\end{assumption}

\begin{assumption}[$L$-smoothness of $\psi_i$]\label{ass:psi}
    There exists a constant $L>0$ such that $\psi_i$ is $L$-smooth for all $i\in\mc I$, i.e., 
    \begin{align*}
    \|\nabla \psi_i(x) - \nabla \psi_i(y)\| \leq L \|x - y\| \quad \forall x, y \in \mathbb{R}^n.
    \end{align*}
\end{assumption}

\section{A Perturbed DCA for Computing d-Stationary Points of the DC Program~\eqref{eq:dc_general}}\label{sec:pDCA}

In this section, we will propose an efficient perturbed difference-of-convex algorithm for solving problem~\eqref{eq:dc_general}, which (subsequentially) converges to d-stationary points of the DC program \eqref{eq:dc_general} almost surely.

Since each $\psi_i$ is continuously differentiable, we define the \textit{active gradient set} of $\psi$ at $x$ as
\begin{align*}
    \mc S(x):=\left\{ \nabla\psi_i(x)\mid i\in\mc M(x) \right\},
\end{align*}
where $\mc M(x)$ is the active index set of $\psi$ at $x$ defined in \eqref{eq:active_set_general}. We have the following result, which is a consequence of Rademacher's theorem \cite{Rademacher1919,rw1998}.

\begin{proposition}
\label{prop:singleton_almost_surely}
Given {$\alpha>0$ and }$x \in \mathbb{R}^n$, let $\xi\in\mb R^n$ be a random vector uniformly distributed on the ball $\{\xi \in \mathbb{R}^n \mid \|\xi\| \leq\alpha\}$
, and $\hat x=x + \xi$. Then, $\mathcal{S}(\hat{x})$ is a singleton almost surely. Furthermore, if each $\psi_i$ is affine, $\mc M(\hat x)$ is also a singleton almost surely.
\end{proposition}

\begin{proof}
Define
$$ \mc D := \left\{ x\in\mb R^n \mid \psi\text{ is differentiable at }x \right\}. $$
It follows from \cite[Thm. 25.1]{rockafellar1970convex} and \cite[Ex. 8.31]{rw1998} that, for any $y\in\mc D$,
\begin{align*}
\partial \psi(y) = \conv\left\{\nabla \psi_i(y) \mid i \in \mathcal{M}(y)\right\} = \{\nabla\psi(y)\}
\end{align*}
is a singleton. Then, we have \(\mc S(y)\) must be a singleton for $y \in \mathcal{D}$. Since each $\psi_i$ is convex and continuously differentiable, $\psi$ is a finite convex function. By Rademacher's theorem \cite{Rademacher1919,rw1998}, we have that
\begin{align*}
\mathbb{P}(\hat{x}\in\mc D) = 1.
\end{align*}
Therefore, \(\mathcal{S}(\hat{x})\) is a singleton almost surely.

Next, we prove the second claim. Since $\psi_i$ are affine functions, for all $i\in \mc I$, \(\psi_i(\cdot) = \langle a_i, \cdot\rangle + b_i\) for some $a_i\in\mb R^n$ and $b_i\in\mb R$ with \(\nabla \psi_i(x) = a_i\) for all $x\in\mb R^n$.   If \(\mathcal{S}(\hat{x})\) is a singleton, we have that $a_i = a_j$ for any \(i,j \in \mathcal{M}(\hat{x})\). Moreover, since \(\psi_i(\hat{x}) = \psi_j(\hat{x})\) for all \(i,j \in \mathcal{M}(\hat{x})\), we have $b_i = b_j$. Since $\psi_i$ are distinct, we know that \(\mathcal{M}(\hat{x})\) is a singleton almost surely.
\end{proof}

Based on Proposition \ref{prop:singleton_almost_surely}, we propose a perturbed difference-of-convex algorithm for solving problem~\eqref{eq:dc_general} in Algorithm~\ref{alg:pDCA}.
\begin{algorithm}[htbp]
\caption{Perturbed DCA for solving \eqref{eq:dc_general}}
\label{alg:pDCA}
\begin{algorithmic}[1]
\setlength{\abovedisplayskip}{3pt}
\setlength{\belowdisplayskip}{3pt}
\REQUIRE $x^0\in\dom(\zeta)$, $\sigma>0$.
\FOR{$k=0,1,\cdots$}
    \STATE \textbf{Step 1}. Randomly sample $\xi^k\sim\mr{Unif}(\mb S^{n-1})$, select $\alpha_k>0$, and construct
    $$e^k = \alpha_k\xi^k, \quad \hat x^k = x^k + e^k.$$ 
    \STATE {\textbf{Step 2}. Randomly select $i_k\in\mc M(\hat{x}^k)$ and set $g^k=\nabla\psi_{i_k}(\hat{x}^k)$.}
    
    \STATE \textbf{Step 3}. Update 
    \begin{equation}
    \label{eq: update_x_pdca}
        x^{k+1}=\argmin\limits_{x\in\mathbb R^n}\left\{\phi(x) - \psi(\hat{x}^k) - \langle g^k,x-\hat x^k \rangle + \dfrac{\sigma}{2}\|x-\hat x^k\|^2 \right\}.
    \end{equation}
\ENDFOR
\end{algorithmic}
\end{algorithm}

\begin{remark}
    In contrast to Algorithm~\ref{alg:Pang}, which at each iteration $k$ requires solving multiple subproblems for all indices in the $\epsilon$-active index set $\mathcal{M}_\epsilon(x^k)$ and selecting the best candidate in terms of the ``proximal objective value'', the proposed perturbed DCA in Algorithm~\ref{alg:pDCA} only solves a single subproblem at each iteration. This substantially reduces the per-iteration computational cost, making the perturbed DCA scalable for solving large-scale nonsmooth DC programs.     {We also emphasize that the random perturbation in perturbed DCA plays a different role from randomly drawing an index from an $\mathcal{M}_\epsilon(x^k)$, e.g., in \cite[Sec.~5.2]{pang2017computing} and \cite[Alg.~2]{van2019non}. Random index selection requires that the relevant relaxed active set be constructed or maintained. It is expensive when $\mathcal{M}_\epsilon(x^k)$ (or $\mathcal{M}(x^k)$) is large. By contrast, perturbed DCA randomly samples a nearby point at which the active-gradient set is a singleton without constructing or maintaining any relaxed active sets.}
\end{remark}

{
\begin{remark}
    Although Algorithm~\ref{alg:pDCA} randomly selects an index \(i_k\in\mathcal M(\hat{x}^k)\) in \textbf{Step 2}, Proposition~\ref{prop:singleton_almost_surely} implies that \(\mathcal S(\hat{x}^k)\) is a singleton almost surely. Hence, all active indices at \(\hat{x}^k\) yield the same gradient almost surely, and \(g^k\) is the unique active gradient. In particular, no resampling of the perturbation is needed.
\end{remark}}

We make the following assumption regarding the perturbations in Algorithm~\ref{alg:pDCA}.

\begin{assumption}\label{ass:perturbation}
The sequence $\{\alpha_k\}_{k \ge 0}$ in Algorithm~\ref{alg:pDCA} satisfies  $\sum_{k=0}^{\infty} \alpha_k^2 < \infty$.
\end{assumption}

Let $\mathcal{F}_k := \sigma\bigl(x^0, \xi^0, x^1, \xi^1, \dots, \xi^{k-1}, x^k\bigr)$ be the $\sigma$-algebra generated by all the randomness up to (and including) $x^k$, and denote $\mb E_k[\cdot]:=\mb E[\cdot\mid\mc F_k]$.  The Lemma~\ref{lem:boundedness} below shows the almost sure boundedness of the sequence generated by Algorithm~\ref{alg:pDCA}.

\begin{lemma}[Sequence boundedness]\label{lem:boundedness}
    Suppose that Assumptions~\ref{ass:coercive}--\ref{ass:perturbation} hold. Let $\{\hat x^k\}$ and $\{x^k\}$ be generated by Algorithm~\ref{alg:pDCA}. Then the following statements hold almost surely:
    \begin{enumerate}
        \item[(i)] The sequence $\{x^k\}$ is bounded;
        \item[(ii)] $\sum_k\mb E_k[\|x^{k+1}-\hat x^k\|^2]<\infty$.
    \end{enumerate}
\end{lemma}
\begin{proof}
Fix any $i \in \mathcal{M}(\hat{x}^k)$. Then, we have 
\begin{align*}
\psi(\hat{x}^k) = \psi_i(\hat{x}^k), \quad g^k = \nabla \psi_i(\hat{x}^k),
\end{align*}
and
\begin{align*}
    \zeta(x^{k+1}){=}& \phi(x^{k+1}) - \psi(x^{k+1})\\
    \leq& \phi(x^{k+1}) - \psi_i(x^{k+1}) \quad(\psi=\max\psi_i) \\
    \leq& \phi(x^{k+1})-\psi(\hat x^k)-\langle g^k,x^{k+1}-\hat x^k\rangle \quad{(\psi_i\text{ convex})} \\
    \leq& \phi(x^k)-\psi(\hat x^k)-\langle g^k,x^k-\hat x^k\rangle+\frac\sigma2\|x^k-\hat x^k\|^2-\frac\sigma2\|x^{k+1}-\hat x^k\|^2 \quad(\text{by }\eqref{eq: update_x_pdca})\\
    =& \phi(x^k)-\psi_{i}(\hat x^k)-\langle\nabla\psi_{i}(\hat x^k), x^k-\hat x^k\rangle +\frac\sigma2\|x^k-\hat x^k\|^2-\frac\sigma2\|x^{k+1}-\hat x^k\|^2 \\
    =& \underbrace{\phi(x^k) - \psi(x^k)}_{=\zeta(x^k)} + \psi(x^k) - \psi(\hat x^k) + \underbrace{\psi(\hat x^k) - \psi_{i}(\hat x^k)}_{=0} + \langle\nabla\psi_{i}(\hat x^k), \hat x^k-x^k\rangle \\ 
    &+ \underbrace{\frac\sigma2\|x^k-\hat x^k\|^2}_{=\frac\sigma2\|e^k\|^2}-\frac\sigma2\|x^{k+1}-\hat x^k\|^2.
\end{align*}
Taking the conditional expectation $\mb E_k[\cdot]$ on both sides of the above inequality, noticing $x^k$ is $\mc F_k$-measurable, we have that
\begin{align}\label{pf:lem-1}
    \notag \mb E_k[\zeta(x^{k+1})] \leq& \zeta(x^k) + \mb E_k[\psi(x^k) - \psi(\hat x^k)] + \mb E_k[\langle\nabla\psi_{i}(\hat x^k), \hat x^k-x^k\rangle] \\ 
    &+ \frac\sigma2\mb E_k[\|e^k\|^2]-\frac\sigma2\mb E_k[\|x^{k+1}-\hat x^k\|^2].
\end{align}
For the second term on the right-hand-side (R.H.S.) of \eqref{pf:lem-1}, since $\psi$ is convex, by Jensen's inequality \cite[Thm.~1.6.2]{durrett2019probability} and the fact that $\mb E_k[e^k]=\alpha_k\mb E_k[\xi^k]=0$, it holds that
\begin{align}\label{pf:lem-2}
    \mb E_k[\psi(x^k) - \psi(\hat x^k)] = \mb E_k[\psi(x^k)] - \mb E_k[\psi(\hat x^k)] \leq \psi(x^k) - \psi(\mb E_k[x^k+e^k]) = 0.
\end{align}
For the third term on the R.H.S. of \eqref{pf:lem-1}, by the $L$-smoothness of $\psi_i$ and the fact that $\mb E_k[e^k]=0$, we have that
\begin{align}\label{pf:lem-3}
    \notag\mb E_k[\langle\nabla\psi_{i}(\hat x^k), \hat x^k-x^k\rangle] =& \mb E_k[\langle\nabla\psi_{i}(\hat x^k)-\nabla\psi_i(x^k), \hat x^k-x^k\rangle] + \mb E_k[\langle \nabla\psi_i(x^k), \hat x^k-x^k \rangle] \\
    \leq& L\mb E_k[\|e^k\|^2].
\end{align}
Substituting \eqref{pf:lem-2} and \eqref{pf:lem-3} into \eqref{pf:lem-1}, we obtain that
\begin{align*}
    \notag\mb E_k[\zeta(x^{k+1})] \leq& \zeta(x^k) + \left(L+\frac\sigma2\right)\mb E_k[\|e^k\|^2] - \frac\sigma2\mb E_k[\|x^{k+1}-\hat x^k\|^2] \\
    =& \zeta(x^k) + \left(L+\frac\sigma2\right){\alpha_k^2} - \frac\sigma2\mb E_k[\|x^{k+1}-\hat x^k\|^2].
\end{align*}
Since {$\{\alpha_k^2\}$} is summable {according to Assumption~\ref{ass:perturbation}}, by the Robbins--Siegmund {theorem \cite[Thm.~1]{robbins1971rslemma}}, $\zeta(x^k)$ converges to a finite random variable and $\sum_k\mb E_k[\|x^{k+1}-\hat x^k\|^2]<\infty$ almost surely. {Therefore, there exists an event \(\Omega_0\) with probability one such that, for every \(\omega\in\Omega_0\), the sequence \(\{\zeta(x^k(\omega))\}\) is bounded above, i.e., \(C(\omega):=\sup_{k\geq0}\zeta(x^k(\omega))<+\infty\). By Assumption~\ref{ass:coercive}, the level set \(\{x\mid \zeta(x)\leq C(\omega)\}\) is bounded. Since \(x^k(\omega)\) belongs to this level set for all \(k\), the sequence \(\{x^k(\omega)\}\) is bounded. Hence, \(\{x^k\}\) is bounded almost surely.}
\end{proof}

\begin{remark}
Due to the Robbins--Siegmund theorem and our selection of random perturbations $\{e^k\}$, we do not require $\{\zeta(x^k)\}_{k\geq0}$ to be monotone decreasing. Therefore, the choice of the parameter $\sigma > 0$ can be independent of the smoothness constant $L$ in Assumption~\ref{ass:psi}. This provides flexibility for us to choose $\sigma$. In Algorithm~\ref{alg:pDCA}, we can choose a sequence of $\{\sigma_k\}_{k \geq 0}$ satisfying $0 < \underline{\sigma} \leq \sigma_k \leq \bar{\sigma} < +\infty$ for some $\underline{\sigma}$ and $\bar{\sigma}$, $\forall k\geq0$. 
\end{remark}

Based on Lemma~\ref{lem:boundedness}, the sequence $\{x^k\}$ generated by Algorithm~\ref{alg:pDCA} possesses at least one cluster point $x^*$ almost surely. We introduce the following assumption regarding the geometry of the subdifferential of $\psi(\cdot)$ at $x^*$, which can guarantee that $x^*$ is a d-stationary point of \eqref{eq:dc_general} almost surely.

\begin{assumption}[Non-emptiness of the exposed direction sets]
\label{ass:linear_independence}
{Let $x^*$ be any candidate cluster point of the sequence generated by the algorithm under consideration.} For any active index $i \in \mathcal{M}(x^*)$, the set of exposed directions, defined by
\begin{align*}
\mathcal{A}_i^* := \left\{ d \in \mathbb{S}^{n-1} \mid \langle \nabla \psi_i(x^*) - \nabla \psi_j(x^*), d \rangle > 0, \forall j \in \mathcal{M}(x^*) \setminus \{i\} \right\}
\end{align*}
is non-empty. {When \(\mathcal{M}(x^*)=\{i\}\), we have \(\mathcal{A}_i^*=\mathbb{S}^{n-1}\neq\emptyset\).}
\end{assumption}

\begin{remark}
\label{remark: sep_assump3}
It follows from the separation theorem \cite[Cor.~11.4.2]{rockafellar1970convex} that, for any $i \in \mathcal{M}(x^*)$, $\mathcal{A}^*_i$ is nonempty if $\nabla\psi_i(x^*)\notin\conv\left(\{\nabla\psi_j(x^*)\}_{j \in \mathcal{M}(x^*)\backslash \{i\}}\right)$.
{Thus, Assumption~\ref{ass:linear_independence} is implied by the condition that every active gradient is an exposed point of the convex hull of the active gradients. 
}
\end{remark}

Assumption~\ref{ass:linear_independence} is not restrictive. Before we prove the theoretical guarantees of perturbed DCA for computing d-stationary points of \eqref{eq:dc_general}, we first show that Assumption~\ref{ass:linear_independence} holds for the following examples we discussed in Section~\ref{sec:DC}.  {If \(\mathcal{M}(x^*)\) is a singleton, Assumption~\ref{ass:linear_independence} holds trivially. Hence, we only need to consider the case \(|\mathcal{M}(x^*)|\geq2\).}

\paragraph{Example 1 ($K$-sparse regularization)} Without loss of generality, we assume $\lambda=1$ in \eqref{eq:dc_ksparse} for simplicity.
An index $i \in \mathcal{M}(x^*)$ corresponds to a subset $\mathcal{J}_i \subset [n]$ with $|\mathcal{J}_i|=K$, representing the indices of the $K$ largest absolute coordinates of $x^*$. Since $\|\nabla \psi_i(x^*)\|^2 = K$ for all $i\in\mc M(x^*)$, we can choose $d = \nabla \psi_i(x^*) / \sqrt{K}$. For any $j \in \mathcal{M}(x^*) \setminus \{i\}$, the fact that $\mc J_j \neq \mc J_i$ implies $\langle \nabla \psi_j(x^*), \nabla \psi_i(x^*) \rangle < K$. Consequently, $\langle \nabla \psi_i(x^*) - \nabla \psi_j(x^*), d \rangle > 0$, ensuring $\mathcal{A}_i^* \neq \emptyset$.

\paragraph{Example 2 (Capped $\ell_1$ penalty)} Recall that for each index $s \in \mathcal{I} := \{-1,0,1\}^n$, $\psi_s(x) = \sum_{j \in \mathrm{supp}(s)} (s_j x_j - \theta)$ (assuming $\lambda=1$) with $\nabla \psi_s(x) = s$. An index $s \in \mathcal{M}(x^*)$ if $s_j = \mathrm{sgn}(x_j^*)$ for $|x_j^*| > \theta$, $s_j = 0$ for $|x_j^*| < \theta$, and $s_j \in \{\mathrm{sgn}(x_j^*), 0\}$ for $|x_j^*| = \theta$. For any given $s \in \mathcal{M}(x^*)$, we construct a direction $d\in\mb R^n$ component-wise as:
\begin{align*}
d_j = \begin{cases} \mathrm{sgn}(x_j^*) & \text{if } s_j \neq 0, \\ -\mathrm{sgn}(x_j^*) & \text{if } s_j = 0. \end{cases}
\end{align*}
For any distinct active index $s' \in \mathcal{M}(x^*) \setminus \{s\}$, it holds that $(s_j - s'_j)d_j = 1$ if $s_j \neq s'_j$, and $(s_j - s'_j)d_j = 0$ if $s_j = s'_j$. Consequently, $\langle \nabla \psi_s(x^*) - \nabla \psi_{s'}(x^*), d/\|d\| \rangle > 0$, which guarantees $\mathcal{A}_s^* \neq \emptyset$ for all $s \in \mathcal{M}(x^*)$.

\paragraph{Example 3 ($K$-medians clustering)} 
Each $i \in \mathcal{M}(\mu^*)$ corresponds to an assignment mapping $\pi_i: [n] \to [K]$ where $\pi_i(k) \in \argmin_{j \in [K]} \{\|\mu_j^* - a_k\|_1\}$. The complete gradient vector $\nabla \psi_i(\mu^*)$ is defined as:
\begin{align*} 
\nabla \psi_i(\mu^*) = \left(\nabla_{\mu_1} \psi_i(\mu^*), \nabla_{\mu_2} \psi_i(\mu^*), \ldots, \nabla_{\mu_K} \psi_i(\mu^*)\right) \in \mathbb{R}^{K \times d}
\end{align*}
where for each cluster $j$ and coordinate $r$:
\begin{align*} 
\nabla_{\mu_j^{(r)}} \psi_i(\mu^*) = \underbrace{\frac{1}{n}\sum_{k=1}^n \mathrm{sgn}(\mu_j^{*(r)} - a_k^{(r)})}_{C_j^{(r)}} - \underbrace{\frac{1}{n}\sum_{k:\pi_i(k)=j} \mathrm{sgn}(\mu_j^{*(r)} - a_k^{(r)})}_{\Delta_{i,j}^{(r)}}.
\end{align*}
The assignment-specific vector is:
\begin{align*} 
\mathbf{\Delta}_i^{(r)} = (\Delta_{i,1}^{(r)}, \Delta_{i,2}^{(r)}, \ldots, \Delta_{i,K}^{(r)}).
\end{align*}

Since the $K$-medians clustering problem is separable over coordinates, it is sufficient for us to consider for any fixed coordinate $r \in [d]$. Moreover, since $C_j^{(r)}$ is a constant over different indices $i \in \mathcal{M}(\mu^*)$, we only need to focus on the values of $\Delta_{i,j}^{(r)}$. Without loss of generality, we assume that $\mu_1^{*(r)} < \mu_2^{*(r)} < \cdots < \mu_K^{*(r)}$. Here, we assume the following \emph{coordinate-wise boundary uniqueness condition}: $\forall 1 \leq j \leq K-1$,
\begin{align*}
\left|\mathcal{B}_{j,j+1}^{(r)}\right| \leq 1 \quad \text{where} \quad \mathcal{B}_{j,j+1}^{(r)} = \left\{a_k \in \mathbb{R}^d : \left|a_k^{(r)} - \mu_j^{*(r)}\right| = \left|a_k^{(r)} - \mu_{j+1}^{*(r)}\right|\right\}.
\end{align*}
If $b \in \mathcal{B}_{j,j+1}^{(r)}$ for some $1 \leq j \leq K-1$, we have
\begin{align*}
\mu_j^{*(r)} < b^{(r)} < \mu_{j+1}^{*(r)}, \quad \mathrm{sgn}(\mu_j^{*(r)} - b^{(r)}) = -1 \quad \text{and} \quad \mathrm{sgn}(\mu_{j+1}^{*(r)} - b^{(r)}) = 1.
\end{align*}
We first consider the case when $K = 2$. If $\mathcal{B}_{1,2}^{(r)} = \emptyset$, Assumption \ref{ass:linear_independence} holds trivially. Here {and} below, we only consider the cases where $\left|\mathcal{B}_{1,2}^{(r)}\right| = 1$. There are two different assignments $\pi_{i_1}$ and $\pi_{i_2}$:
\begin{align*}
i_1: & \quad \pi_{i_1}(b) = 1 \quad \Rightarrow \quad \mathbf{\Delta}_{i_1}^{(r)} = \left(-\frac{1}{n}, 0\right), \\
i_2: & \quad \pi_{i_2}(b) = 2 \quad \Rightarrow \quad \mathbf{\Delta}_{i_2}^{(r)} = \left(0, \frac{1}{n}\right).
\end{align*}
In this case, Assumption~\ref{ass:linear_independence} holds directly according to Remark \ref{remark: sep_assump3}.

Now, we consider the case $K = 3$ with two boundary points $b_1 \in \mathcal{B}_{1,2}^{(r)}$ and $b_2 \in \mathcal{B}_{2, 3}^{(r)}$. There are four different assignments $\{\pi_{i_1}, \pi_{i_2}, \pi_{i_3}, \pi_{i_4}\}$:
\begin{align*}
i_1: & \quad \pi_{i_1}(b_1) = 1, \pi_{i_1}(b_2) = 2 \quad \Rightarrow \quad \mathbf{\Delta}_{i_1}^{(r)} = \left(-\frac{1}{n}, -\frac{1}{n}, 0\right), \\
i_2: & \quad \pi_{i_2}(b_1) = 2, \pi_{i_2}(b_2) = 2 \quad \Rightarrow \quad \mathbf{\Delta}_{i_2}^{(r)} = \left(0, 0, 0\right), \\
i_3: & \quad \pi_{i_3}(b_1) = 1, \pi_{i_3}(b_2) = 3 \quad \Rightarrow \quad \mathbf{\Delta}_{i_3}^{(r)} = \left(-\frac{1}{n}, 0, \frac{1}{n}\right), \\
i_4: & \quad \pi_{i_4}(b_1) = 2, \pi_{i_4}(b_2) = 3 \quad \Rightarrow \quad \mathbf{\Delta}_{i_4}^{(r)} = \left(0, \frac{1}{n}, \frac{1}{n}\right).
\end{align*}
We can naturally classify the four assignments into two disjoint groups by the value of the first component:\\
\textbf{Group A}: $i_1$ and $i_3$:
\begin{align*}
\mathbf{\Delta}_{i_1}^{(r)} = \left(-\frac{1}{n}, -\frac{1}{n}, 0\right), \quad 
\mathbf{\Delta}_{i_3}^{(r)} = \left(-\frac{1}{n}, 0, \frac{1}{n}\right)
\end{align*}
\textbf{Group B}: $i_2$ and $i_4$
\begin{align*}
\mathbf{\Delta}_{i_2}^{(r)} = \left(0, 0, 0\right), \quad
\mathbf{\Delta}_{i_4}^{(r)} = \left(0, \frac{1}{n}, \frac{1}{n}\right).
\end{align*}
It is straightforward to verify that
\begin{align*}
{\rm conv}\left(\{\mathbf{\Delta}_{i_1}^{(r)}, \mathbf{\Delta}_{i_3}^{(r)}\}\right) \bigcap {\rm conv}\left(\{\mathbf{\Delta}_{i_2}^{(r)}, \mathbf{\Delta}_{i_4}^{(r)}\}\right) = \emptyset.
\end{align*}
Moreover, within each group, it reduces to the case for $K = 2$. Therefore, it follows from Remark \ref{remark: sep_assump3} that Assumption \ref{ass:linear_independence} holds for the case $K = 3$. 

By an induction argument, we can similarly generalize the discussion above for the cases with $K > 3$. Therefore, we verified Assumption~\ref{ass:linear_independence} under the coordinate-wise boundary uniqueness condition for the $K$-medians clustering problem.

We now move on to establish the theoretical guarantees of the perturbed DCA. If $\mathcal M(x^*)$ is not a singleton, define
\begin{equation*}
\begin{aligned}
    a_{ij}&:=\nabla\psi_i(x^*)-\nabla\psi_j(x^*),\quad
    \beta_i:=\max_{\xi\in\mathbb S^{n-1}}
    \min_{j\in\mathcal M(x^*)\setminus\{i\}}
    \langle a_{ij},\xi\rangle, \\
    D_i&:=\max_{j\in\mathcal M(x^*)\setminus\{i\}}
    \|a_{ij}\|,\text{ and} \quad
    \bar r:=\min_{i\in\mathcal M(x^*)}\frac{\beta_i}{D_i}.
\end{aligned}
\end{equation*}

The Lemma~\ref{lem:selection_prob} below is one of the key theoretical results in this paper, which proves a uniform lower bound on the probability of selecting any active gradient in $\mathcal{S}(x^*)$ at a cluster point $x^*$ under the perturbation mechanism of Algorithm~\ref{alg:pDCA}.

\begin{lemma}[Uniform lower bound on selection probability]
\label{lem:selection_prob}
Suppose that Assumptions~\ref{ass:coercive}--\ref{ass:linear_independence} hold. Let $\{x^k\}$ and $\{\hat{x}^k\}$ be the sequences generated by Algorithm~\ref{alg:pDCA}, $x^*$ be a cluster point of $\{x^k\}$ and $\{x^k\}_{k \in \mathcal{K}}$ be a subsequence converging to $x^*$. {If either $\mathcal M(x^*)$ is a singleton or there exist $C\in(0,\bar r)$ and an infinite subset $\mc K_{\mr{sub}}\subseteq\mathcal K$ satisfying
\begin{align}\label{eq:rate-radii-ratio}
    \frac{\|x^k-x^*\|}{\alpha_k}\le C
\end{align}
for all sufficiently large $k\in\mc K_{\mr{sub}}$},
then there exists an infinite subset $\mathcal{K}' \subseteq \mathcal{K}$ and a constant $p_{\min} > 0$ such that for any $i \in \mathcal{M}(x^*)$ and all $k \in \mathcal{K}'$:
\begin{align*}
\mathbb{P}\left(\mathcal{S}(\hat{x}^k) = \{\nabla \psi_i(\hat{x}^k)\} \mid \mathcal{F}_k \right) \geq p_{\min} > 0.
\end{align*}
\end{lemma}

\begin{proof}
First, we show $\mathcal{M}(\hat{x}^k) \subseteq \mathcal{M}(x^*)$ for all sufficiently large $k{\in\mc K}$. {For any \(j\notin \mc M(x^*)\) and \(i\in \mc M(x^*)\),
\begin{align*}
    \psi_i(x^*)>\psi_j(x^*).
\end{align*}
Because \(x^k\to x^*\) along \(\mathcal K\) and \(\alpha_k\to0\), we have
\begin{align*}
    \hat x^k=x^k+\alpha_k\xi^k\to x^*
    \quad
    \text{along } \mathcal K.
\end{align*}
Thus, there exists $K_1>0$ such that for all $k\geq K_1$ with \(k\in\mathcal K\),
\begin{align}\label{pf-active-inactive-gap}
    \psi_i(\hat x^k)>\psi_j(\hat x^k),
    \qquad
    \forall j\notin \mc M(x^*),
\end{align}}
ensuring that $\mathcal{M}(\hat{x}^k) \subseteq \mathcal{M}(x^*)$.

{If \(\mathcal{M}(x^*)=\{i\}\), then we have \(\mathcal{M}(\hat{x}^k)=\{i\}\) for all $k\in\mc K$ with $k\geq K_1$. Therefore,
\begin{align*}
    \mathbb{P}\left(\mathcal{S}(\hat{x}^k)=\{\nabla\psi_i(\hat{x}^k)\}\mid\mathcal{F}_k\right)=1
\end{align*}
for all such \(k\), and the desired conclusion holds with \(p_{\min}=1\). It remains to consider the case \(|\mathcal{M}(x^*)|\geq2\).}

We fix any $i \in \mathcal{M}(x^*)$. { Since \(C<\bar r\le \beta_i/D_i\), we have
\begin{align*}
    \widetilde\eta_i:=\beta_i-D_iC>0.
\end{align*}
Choose \(\xi_i\in\mb S^{n-1}\) such that
\begin{align*}
    \min_{j\in \mc M(x^*)\setminus\{i\}}
    \langle a_{ij},\xi_i\rangle
    =
    \beta_i.
\end{align*}
Define the spherical cap
\begin{align*}
    \widetilde B_i
    :=
    \left\{
        \xi\in\mb S^{n-1}:
        \|\xi-\xi_i\|
        \le
        \frac{\widetilde\eta_i}{4D_i}
    \right\}.
\end{align*}
Since \(\widetilde B_i\) has positive surface measure, there exists
\begin{align*}
    \widetilde p_i:=\mb P(\xi^k\in\widetilde B_i)>0.
\end{align*}
For any \(\xi\in\widetilde B_i\) and any \(j\in \mc M(x^*)\setminus\{i\}\),
\begin{equation}\label{pf-0}
\begin{aligned}
    \langle a_{ij},\xi\rangle
    &\ge
    \langle a_{ij},\xi_i\rangle
    -
    \|a_{ij}\|\,\|\xi-\xi_i\|
    \\
    &\ge
    \beta_i
    -
    D_i\frac{\widetilde\eta_i}{4D_i}
    =
    \beta_i-\frac{\widetilde\eta_i}{4}
    =
    D_iC+\frac{3\widetilde\eta_i}{4}.
\end{aligned}
\end{equation}
Let
\(
    d_k:=x^k-x^*.
\)
For \(i,j\in \mc M(x^*)\), we have
\(    
    \psi_i(x^*)=\psi_j(x^*).
\)
By the \(L\)-smoothness of \(\psi_i\) and \(\psi_j\),
\begin{equation}\label{pf-1}
\begin{aligned}
    \psi_i(\hat x^k)-\psi_j(\hat x^k)
    &\ge
    \langle a_{ij},d_k+\alpha_k\xi^k\rangle
    -
    L\|d_k+\alpha_k\xi^k\|^2
    \\
    &\ge
    \alpha_k\langle a_{ij},\xi^k\rangle
    -
    D_i\|d_k\|
    -
    L(\|d_k\|+\alpha_k)^2
    \\
    &=
    \alpha_k
    \left[
        \langle a_{ij},\xi^k\rangle
        -
        D_i\frac{\|d_k\|}{\alpha_k}
        -
        L\alpha_k
        \left(
            1+\frac{\|d_k\|}{\alpha_k}
        \right)^2
    \right].
\end{aligned}
\end{equation}
For sufficiently large \(k\in\mc K_{\mr{sub}}\), condition~\eqref{eq:rate-radii-ratio} gives
\begin{align*}
    \frac{\|d_k\|}{\alpha_k}\le C.
\end{align*}
If moreover \(\xi^k\in\widetilde B_i\), then by \eqref{pf-0} and \eqref{pf-1},
\begin{align*}
    \psi_i(\hat x^k)-\psi_j(\hat x^k)
    &\ge
    \alpha_k
    \left[
        D_iC+\frac{3\widetilde\eta_i}{4}
        -
        D_iC
        -
        L\alpha_k(1+C)^2
    \right].
\end{align*}
	Since \(\alpha_k\to0\), there exists $K_{2,i}>0$ such that for $k\geq K_{2,i}$ with \(k\in\mc K_{\mr{sub}}\),
\begin{align*}
    L\alpha_k(1+C)^2
    \le
    \frac{\widetilde\eta_i}{4}.
\end{align*}
Hence
\begin{align}\label{pf-select-specific-i}
    \psi_i(\hat x^k)-\psi_j(\hat x^k)
    \ge
    \frac{\widetilde\eta_i}{2}\alpha_k
    >
    0,
    \qquad
    \forall j\in \mc M(x^*)\setminus\{i\}.
\end{align}
Combining \eqref{pf-active-inactive-gap} and \eqref{pf-select-specific-i}, we obtain
\begin{align*}
    \mathcal{M}(\hat x^k)=\{i\}
\end{align*}
	whenever $k\in\mathcal{K}'_i := \{k \in \mathcal{K}_{\mr{sub}} \mid k \ge \max\{K_1, K_{2,i}\}\}$ and \(\xi^k\in\widetilde B_i\). Therefore,
\begin{align*}
    \mathcal{S}(\hat x^k)=\{\nabla\psi_i(\hat x^k)\}.
\end{align*}
Consequently,
\begin{align*}
    \mb P\left(
        \mathcal{S}(\hat x^k)=\{\nabla\psi_i(\hat x^k)\}
        \mid\mathcal F_k
    \right)
    \ge
    \mb P(\xi^k\in\widetilde B_i\mid\mathcal F_k)
    =
    \mb P(\xi^k\in\widetilde B_i)
    =
    \widetilde p_i>0.
\end{align*}
}
	{Since \(\mathcal M(x^*)\) is finite, we can take \(K_2:=\max_{i\in\mathcal M(x^*)}K_{2,i}\), \(\mathcal K':=\{k\in\mathcal K_{\mr{sub}}\mid k\ge \max\{K_1, K_2\}\}\), and \(p_{\min}:=\min_{i\in\mathcal M(x^*)}\widetilde p_i>0\). Then the same infinite index set \(\mathcal K'\) and the same lower bound \(p_{\min}\) work for all \(i\in\mathcal M(x^*)\), which completes the proof.}
\end{proof}

\begin{remark}
    Indeed, Assumption~\ref{ass:linear_independence} is also necessary for guaranteeing $p_{\min} > 0$ in Lemma~\ref{lem:selection_prob} in general. To see this, consider the counterexample with $\psi(x) = \max\{x, -x, 0\}$ at $x^* = 0$. 
\end{remark}

\begin{remark}
    Intuitively, the perturbation should not shrink so fast that it misses some active pieces near $x^*$. The relative perturbation condition~\eqref{eq:rate-radii-ratio} ensures that the random perturbation is able to reach every active region around the cluster point with positive probability. This condition can be satisfied when the iterates approach $x^*$ at least as the perturbation radius decays.
\end{remark}

Now, we are ready to prove the main convergence guarantee of perturbed DCA for solving the nonsmooth DC program~\eqref{eq:dc_general} in Theorem~\ref{thm:main}. 

\begin{theorem}
\label{thm:main}
Suppose that Assumptions~\ref{ass:coercive}--\ref{ass:linear_independence} hold. Let \(\{x^k\}\) be the sequence generated by Algorithm~\ref{alg:pDCA}. 
With probability one, every cluster point $x^*$ for which the conditions of Lemma~\ref{lem:selection_prob} hold is a d-stationary point of $\zeta$.
\end{theorem}

\begin{proof}

From Lemma~\ref{lem:boundedness}-(ii), we have
\begin{align*}
\sum_{k=0}^{\infty} \mathbb{E}_k[\|x^{k+1} - \hat{x}^k\|^2] < \infty \quad \text{a.s.}.
\end{align*}
By \cite[Thm.~1.6.4]{durrett2019probability}, we have that: $\forall\varepsilon>0$,
\begin{align*}
    \mb P\left(\|x^{k+1} - \hat{x}^k\|\geq\varepsilon\mid\mc F_k\right) \leq \frac{\mb E_k[\|x^{k+1} - \hat{x}^k\|^2]}{\varepsilon^2}.
\end{align*}
Thus, $\sum_k\mb P(\|x^{k+1} - \hat{x}^k\|\geq\varepsilon\mid\mc F_k)<\infty$, which implies $\|x^{k+1} - \hat{x}^k\|\to0$ almost surely by Borel--Cantelli lemma \cite[Thm.~2.3.1]{durrett2019probability}.  Thus, we have
\begin{align*}
\|x^{k+1} - x^k\| \le \|x^{k+1} - \hat{x}^k\| + \|e^k\| \to 0 \quad \text{a.s. as $k \to \infty$}.
\end{align*}

By Lemma~\ref{lem:boundedness}-(i), the sequence $\{x^k\}$ is bounded almost surely. Let $x^*$ be a cluster point of $\{x^k\}$, and consider a subsequence $\{x^k\}_{k \in \mathcal{K}}$ converging to $x^*$. By Lemma~\ref{lem:selection_prob}, there exists an infinite subset $\mathcal{K}' \subseteq \mathcal{K}$ and $p_{\min}>0$ such that for each $i \in \mathcal{M}(x^*)$,
$$\mathbb{P}\left(\mathcal{S}(\hat{x}^k) = \{\nabla \psi_i(\hat{x}^k)\} \mid \mathcal{F}_k\right)\geq p_{\min} > 0 \quad  \forall k \in \mathcal{K}'.$$
We fix a full-measure event on which the above subsequence $\mathcal K'$ and the lower bound hold. Conditional on this event, applying the {conditional Borel--Cantelli lemma~\cite[Thm.~5.3.2]{durrett2019probability}} to the events $$E_k^i := \left\{\mathcal{S}(\hat{x}^k) = \{\nabla \psi_i(\hat{x}^k)\}\right\},\qquad k\in\mathcal K',$$ implies that, for each $i \in \mathcal{M}(x^*)$, the event $E_k^i$ occurs infinitely often along $\mathcal K'$ almost surely. Consequently, for any $i \in \mathcal{M}(x^*)$, we can find a further subsequence $\mathcal{K}_i \subseteq \mathcal{K}' \subseteq \mathcal{K}$ such that $\mathcal{S}(\hat{x}^k) = \{\nabla \psi_i(\hat{x}^k)\}$ for all $k \in \mathcal{K}_i$ almost surely. On this subsequence $\{x^k\}_{k \in \mathcal{K}_i}$, by the fact that  $x^{k+1}$ is the unique minimizer of \eqref{eq: update_x_pdca} in Algorithm~\ref{alg:pDCA} for each $k$, we have for all $x \in \mathbb{R}^n$,
\begin{align} \label{eq:optimality_k}
\notag&\phi(x^{k+1}) - \langle \nabla \psi_i(\hat{x}^k), x^{k+1} - \hat{x}^k \rangle + \frac{\sigma}{2} \|x^{k+1} - \hat{x}^k\|^2 \\ \leq& \phi(x) - \langle \nabla \psi_i(\hat{x}^k), x - \hat{x}^k \rangle + \frac{\sigma}{2} \|x - \hat{x}^k\|^2.
\end{align}
Recall that $x^k \xrightarrow{k\in\mathcal{K}_i} x^*$, $\hat{x}^k \xrightarrow{k\in\mathcal{K}_i} x^*$, and $\|x^{k+1} - x^k\| \xrightarrow{k\in\mathcal{K}_i} 0$ almost surely, which implies $x^{k+1} \xrightarrow{k\in\mathcal{K}_i} x^*$ almost surely. Combined with the continuity of $\nabla \psi_i$, it follows that $\nabla\psi_i(\hat{x}^k) \xrightarrow{k\in\mathcal{K}_i} \nabla \psi_i(x^*)$ almost surely. Letting $k \to \infty$ along $\mathcal{K}_i$ on both sides of \eqref{eq:optimality_k}, and using the lower semi-continuity of $\phi$, we obtain:
\begin{align*}
\phi(x^*) \leq \phi(x) - \langle \nabla \psi_i(x^*), x - x^* \rangle + \frac{\sigma}{2} \|x - x^*\|^2 \quad\text{a.s.}\quad \forall x \in \mathbb{R}^n.
\end{align*}
This inequality implies that
\begin{align*}
x^* = \argmin_{x \in \mathbb{R}^n} \left\{ \phi(x) - \langle \nabla \psi_i(x^*), x - x^* \rangle + \frac{\sigma}{2} \|x - x^*\|^2 \right\} \quad\text{a.s.}.
\end{align*}
Since $i \in \mathcal{M}(x^*)$ was chosen arbitrarily, $x^*$ is a d-stationary point of $\zeta$ almost surely according to Proposition~\ref{pro:d-stationary}.
\end{proof}

{
\section{Hybrid Perturbed DCA}\label{sec:hybrid-pDCA}

Lemma~\ref{lem:selection_prob} requires an additional condition~\eqref{eq:rate-radii-ratio}, which couples the perturbation radii with the unknown local convergence behavior of the iterates. To resolve the issue, in this section, we propose a hybrid version of perturbed DCA based on randomized multiscale perturbations. Specifically, at each iteration, both the perturbation radius and direction are sampled independently, with the radius distribution assigning positive probability to every nondegenerate interval near zero. Consequently, along any subsequence converging to a cluster point, arbitrarily fine perturbation scales can be extracted almost surely to reach every relevant active region.

For notational simplicity, for any \(z,g\in\mb R^n\), define
\[
    T(z,g)
    :=
    \argmin_{x\in\mb R^n}
    \left\{
        \phi(x)
        -
        \langle g,x-z\rangle
        +
        \frac{\sigma}{2}\|x-z\|^2
    \right\}.
\]
The resulting hybrid perturbed DCA is presented in Algorithm~\ref{alg:hybrid-pdca}. A candidate $y^k$ is computed by solving the subproblem at the perturbed point. If $y^k$ does not yield sufficient decrease relative to $x^k$, a fallback proximal DCA step is performed. The sufficient-decrease safeguard ensures monotonic descent, while the multiscale sampling mechanism removes both the relative-rate condition and any cluster-point-dependent upper bound on the perturbation radii.

\begin{algorithm}[htbp]
\caption{Hybrid perturbed DCA for solving~\eqref{eq:dc_general}}
\label{alg:hybrid-pdca}
\begin{algorithmic}[1]
\setlength{\abovedisplayskip}{2pt}
\setlength{\belowdisplayskip}{2pt}
\REQUIRE \(x^0\in\dom(\zeta)\), \(\sigma>0\),
\(\eta\in(0,\sigma/2)\), and \(\bar\alpha>0\).
\FOR{\(k=0,1,2,\ldots\)}
    \STATE Sample
    \(\alpha_k\sim\operatorname{Unif}(0,\bar\alpha)\) and
    \(\xi^k\sim\operatorname{Unif}(\mb S^{n-1})\) independently of each
    other, and set
    \(
        \hat x^k=x^k+\alpha_k\xi^k.
    \)
    \STATE Choose any \(i_k\in\mathcal M(\hat x^k)\) and set
    \(
        g^k=\nabla\psi_{i_k}(\hat x^k).
    \)
    \STATE Compute
    \(
        y^k=T(\hat x^k,g^k),
        \
        v_k=\|y^k-\hat x^k\|.
    \)
    \IF{
    \begin{align}
    \label{eq:sufficient-decrease-check}
    \zeta(y^k)\le \zeta(x^k)-\eta v_k^2
    \end{align}
    }
        \STATE Set
        \(
        x^{k+1}=y^k,
        \
        r_k=v_k.
        \)
    \ELSE
        \STATE Choose any \(j_k\in\mathcal M(x^k)\) and set
        \(h^k=\nabla\psi_{j_k}(x^k)\).
        \STATE Compute
        \(
        d^k=T(x^k,h^k).
        \)
        \STATE Set
        \(
        x^{k+1}=d^k,
        \
        r_k=\|d^k-x^k\|.
        \)
    \ENDIF
\ENDFOR
\end{algorithmic}
\end{algorithm}

For Algorithm~\ref{alg:hybrid-pdca}, let
\[
\mathcal G_k
:=
\sigma\bigl(
x^0,\alpha_0,\xi^0,x^1,\ldots,
\alpha_{k-1},\xi^{k-1},x^k
\bigr)
\]
denote the history available immediately before \((\alpha_k,\xi^k)\) is sampled. Notice that \(e^k:=\alpha_k\xi^k\) is not uniformly distributed on a ball. Nevertheless, its distribution is absolutely continuous with respect to the Lebesgue measure on \(\{e\in\mb R^n:\|e\|<\bar\alpha\}\). Hence, by Rademacher's theorem~\cite{Rademacher1919,rw1998}, the active-gradient set \(\mc S(\hat x^k)\) is a singleton almost surely.


We first state the descent property of hybrid perturbed DCA in the following Lemma~\ref{lem:hybrid-descent-boundedness}, which yields boundedness of \(\{x^k\}\) and vanishing residual \(\{r_k\}\).

\begin{lemma}[Sufficient decrease and boundedness]
\label{lem:hybrid-descent-boundedness}
Suppose that Assumption~\ref{ass:coercive} holds. Let \(\{x^k\}\) and \(\{r_k\}\) be generated by Algorithm~\ref{alg:hybrid-pdca}. Then
\begin{equation}
\label{eq:hybrid-uniform-decrease}
\zeta(x^{k+1}) \le \zeta(x^k)-\eta r_k^2,
\qquad \forall k\geq0.
\end{equation}
Consequently, \(\{x^k\}\) is bounded, \(\sum_{k=0}^{\infty}r_k^2<+\infty\), and \(r_k\to0\).
\end{lemma}

\begin{proof}
If the perturbed-point candidate is accepted at iteration \(k\), then by~\eqref{eq:sufficient-decrease-check},
\begin{equation}
\label{eq:hybrid-accept-decrease}
    \zeta(x^{k+1})
    =
    \zeta(y^k)
    \le
    \zeta(x^k)-\eta v_k^2
    =
    \zeta(x^k)-\eta r_k^2.
\end{equation}

If the fallback proximal DCA step is used, then \(x^{k+1}=d^k=T(x^k,h^k)\), where \(h^k=\nabla\psi_{j_k}(x^k)\) and \(j_k\in\mathcal M(x^k)\). Hence, \(\psi(x^k)=\psi_{j_k}(x^k)\). By the convexity of \(\psi_{j_k}\),
\[
\psi_{j_k}(d^k)
\ge
\psi_{j_k}(x^k)+\langle h^k,d^k-x^k\rangle.
\]
Since \(\psi(d^k)\ge\psi_{j_k}(d^k)\), it follows that
\begin{equation}
\label{eq:fallback-psi-bound}
-\psi(d^k)
\le
-\psi_{j_k}(x^k)-\langle h^k,d^k-x^k\rangle.
\end{equation}
Moreover, by the definition of \(d^k\) and the strong convexity of the fallback subproblem,
\[
\phi(d^k)-\langle h^k,d^k-x^k\rangle
+\frac{\sigma}{2}\|d^k-x^k\|^2
\le
\phi(x^k)-\frac{\sigma}{2}\|d^k-x^k\|^2.
\]
Thus,
\begin{equation}
\label{eq:fallback-phi-bound}
\phi(d^k)-\langle h^k,d^k-x^k\rangle
\le
\phi(x^k)-\sigma\|d^k-x^k\|^2
\le
\phi(x^k)-\frac{\sigma}{2}\|d^k-x^k\|^2.
\end{equation}
Combining \eqref{eq:fallback-psi-bound} and \eqref{eq:fallback-phi-bound}, we obtain
\begin{equation}
\label{eq:hybrid-fallback-decrease}
\zeta(x^{k+1})
=
\zeta(d^k)
\le
\zeta(x^k)-\frac{\sigma}{2}\|d^k-x^k\|^2
\le
\zeta(x^k)-\eta r_k^2.
\end{equation}
Equations~\eqref{eq:hybrid-accept-decrease} and \eqref{eq:hybrid-fallback-decrease} prove \eqref{eq:hybrid-uniform-decrease}.

In particular, \(\{\zeta(x^k)\}\) is nonincreasing and
\[
x^k\in\{x\mid \zeta(x)\le\zeta(x^0)\},
\qquad \forall k\ge0.
\]
By Assumption~\ref{ass:coercive}, this level set is bounded, and hence \(\{x^k\}\) is bounded. Furthermore, the lower semicontinuity and coercivity of \(\zeta\) imply that \(\zeta\) is bounded from below. Summing \eqref{eq:hybrid-uniform-decrease} from \(k=0\) to \(N\) gives
\[
\eta\sum_{k=0}^{N}r_k^2
\le
\zeta(x^0)-\zeta(x^{N+1})
\le
\zeta(x^0)-\inf_x\zeta.
\]
Therefore, \(\sum_{k=0}^{\infty}r_k^2<+\infty\), and hence \(r_k\to0\).
\end{proof}

Lemma~\ref{lem:hybrid-descent-boundedness} ensures the existence of the cluster point of \(\{x^k\}\). For each candidate cluster point \(x^*\) and each \(i\in\mathcal M(x^*)\), Assumption~\ref{ass:linear_independence} allows us to choose a compact spherical cap \(B_i\subset\mathcal A_i^*\) and a constant \(\eta_i>0\) such that
\begin{equation}
\label{eq:Bi-margin}
\left\langle
\nabla\psi_i(x^*)-\nabla\psi_j(x^*),d
\right\rangle
\ge \eta_i,
\quad
\forall d\in B_i,\quad
\forall j\in\mathcal M(x^*)\setminus\{i\}.
\end{equation}
Indeed, when \(\mathcal M(x^*)\ne\{i\}\), one may choose a direction in \(\mathcal A_i^*\) and then take a sufficiently small compact spherical cap around it. When \(\mathcal M(x^*)=\{i\}\), we set \(B_i=\mb S^{n-1}\) and regard~\eqref{eq:Bi-margin} as void. 

Define
\[
    p_i:=\mb P(\xi^k\in B_i)>0.
\]
We next establish a local selection property on each fixed scale interval. In particular, it shows that, on a suitable neighborhood of a candidate cluster point, perturbations with directions in $B_i$ select the prescribed active piece uniformly over that interval.

\begin{lemma}[Local multiscale selection]
\label{lem:hybrid-multiscale-selection}
Suppose that Assumptions~\ref{ass:coercive} and \ref{ass:linear_independence} hold, and let \(x^*\) be a cluster point of the sequence \(\{x^k\}\) generated by Algorithm~\ref{alg:hybrid-pdca}. For every \(i\in\mathcal M(x^*)\), there exists \(\bar\rho_i(x^*)>0\) such that, for every fixed $t$ satisfying
\[
0<t<\min\{\bar\rho_i(x^*),\bar\alpha\},
\]
there exists a neighborhood \(\mathcal U_{i,t}\) of \(x^*\) satisfying
\begin{equation}
\label{eq:multiscale-local-dominance}
\psi_i(x+\rho d)>\psi_j(x+\rho d),
\quad
\begin{array}{l}
\forall x\in\mathcal U_{i,t},\quad
\forall \rho\in[t/2,t],\\
\forall d\in B_i,\quad
\forall j\in\mathcal I\setminus\{i\},
\end{array}
\end{equation}
where $B_i$ is defined in \eqref{eq:Bi-margin}.
\end{lemma}

\begin{proof}
Fix \(i\in\mathcal M(x^*)\). For \(\mathcal I=\{i\}\), the conclusion is immediate. We therefore assume that \(\mathcal I\setminus\{i\}\ne\emptyset\). For \(j\in\mathcal M(x^*)\setminus\{i\}\), the integral Taylor formula gives
\[
\begin{aligned}
\psi_i(x^*+\rho d)-\psi_j(x^*+\rho d)=
\int_0^\rho
\left\langle
\nabla\psi_i(x^*+s d)-\nabla\psi_j(x^*+s d),d
\right\rangle\,ds.
\end{aligned}
\]
By~\eqref{eq:Bi-margin}, the continuity of the gradients, the compactness of \(B_i\), and the finiteness of \(\mathcal M(x^*)\), there exists \(\rho_i^{\rm act}>0\) such that
\[
\psi_i(x^*+\rho d)-\psi_j(x^*+\rho d)
\ge \frac{\eta_i}{2}\rho
\]
for all \(0<\rho<\rho_i^{\rm act}\), \(d\in B_i\), and \(j\in\mathcal M(x^*)\setminus\{i\}\).

For every \(j\notin\mathcal M(x^*)\), \(\psi_i(x^*)-\psi_j(x^*)>0\). By continuity, compactness of \(B_i\), and finiteness of \(\mathcal I\), there exists \(\rho_i^{\rm inact}>0\) such that
\[
\psi_i(x^*+\rho d)>\psi_j(x^*+\rho d)
\]
for all \(0<\rho<\rho_i^{\rm inact}\), \(d\in B_i\), and \(j\notin\mathcal M(x^*)\). Set
\[
\bar\rho_i(x^*):=\min\{\rho_i^{\rm act},\rho_i^{\rm inact}\}>0.
\]

Now fix \(t\in(0,\min\{\bar\rho_i(x^*),\bar\alpha\})\). The set \([t/2,t]\times B_i\) is compact, and all the gaps
\[
\psi_i(x^*+\rho d)-\psi_j(x^*+\rho d),
\qquad j\in\mathcal I\setminus\{i\},
\]
are strictly positive on this set. Their minimum over the finitely many indices and the compact set is therefore positive. By the joint continuity of the finitely many functions \((x,\rho,d)\mapsto\psi_i(x+\rho d)-\psi_j(x+\rho d)\), there exists a neighborhood \(\mathcal U_{i,t}\) of \(x^*\) on which all these gaps remain strictly positive. This proves~\eqref{eq:multiscale-local-dominance}.
\end{proof}

The preceding interval-wise property can now be combined with the randomized radii. The next lemma extracts, for every active piece, a subsequence of increasingly fine perturbations that selects that piece almost surely.

\begin{lemma}[Multiscale sampling of every active piece]
\label{lem:hybrid-active-selection}
Suppose that Assumptions~\ref{ass:coercive} and \ref{ass:linear_independence} hold. Let \(\{x^k\}\) be generated by Algorithm~\ref{alg:hybrid-pdca}, \(x^*\) be a cluster point, and \(\{x^k\}_{k\in\mathcal K}\) be the subsequence converges to \(x^*\). Then, for every \(i\in\mathcal M(x^*)\), almost surely, there exists a strictly increasing sequence \(\{k_m\}\subseteq\mathcal K\) such that
\[
\alpha_{k_m}\to0,\qquad
\mathcal M(\hat x^{k_m})=\{i\},\qquad
\hat x^{k_m}\to x^*.
\]
\end{lemma}

\begin{proof}
Fix \(i\in\mathcal M(x^*)\), and choose a deterministic sequence \(t_m\downarrow0\) satisfying
\[
0<t_m<\min\{\bar\rho_i(x^*),\bar\alpha\},
\qquad \forall m\in\mb N,
\]
where \(\bar\rho_i(x^*)\) is given by Lemma~\ref{lem:hybrid-multiscale-selection}. For each fixed \(m\), define
\[
E_k^{i,m}
:=
\left\{
\alpha_k\in[t_m/2,t_m],
\quad
\xi^k\in B_i
\right\}.
\]
Since \(x^k\to x^*\) along \(\mathcal K\), there exists \(K_{i,m}\) such that
\[
x^k\in\mathcal U_{i,t_m},
\qquad
\forall k\in\mathcal K,\quad k\ge K_{i,m}.
\]
Thus, by Lemma~\ref{lem:hybrid-multiscale-selection}, for all such \(k\),
\[
E_k^{i,m}
\quad\Longrightarrow\quad
\mathcal M(\hat x^k)=\{i\}.
\]

By the randomized sampling rule in Algorithm~\ref{alg:hybrid-pdca}, for every fixed \(m\),
\begin{equation*}
\begin{aligned}
\mb P(E_k^{i,m}\mid\mathcal G_k)
=
\mb P\bigl(\alpha_k\in[t_m/2,t_m]\bigr) \mb P(\xi^k\in B_i)
=
\frac{t_m}{2\bar\alpha}\,p_i
=:
q_{i,m}>0.
\end{aligned}
\end{equation*}
Consequently,
\[
\sum_{k\in\mathcal K}
\mb P(E_k^{i,m}\mid\mathcal G_k)
=+\infty
\]
for every fixed \(m\). The conditional Borel--Cantelli lemma~\cite[Thm.~5.3.2]{durrett2019probability} implies that, for every fixed \(m\), \(E_k^{i,m}\) occurs infinitely often along \(\mathcal K\) almost surely. Since \(m\in\mb N\) is countable, these conclusions hold simultaneously for all \(m\) on a common event of probability one.

On this event, recursively choose
\[
k_1<k_2<\cdots,\qquad k_m\in\mathcal K,
\]
such that \(k_m\ge K_{i,m}\) and \(E_{k_m}^{i,m}\) occurs. Then
\[
0<\alpha_{k_m}\le t_m\to0
\]
and \(\mathcal M(\hat x^{k_m})=\{i\}\). Finally,
\[
\|\hat x^{k_m}-x^*\|
\le
\|x^{k_m}-x^*\|+\alpha_{k_m}\to0,
\]
which completes the proof.
\end{proof}

\begin{remark}
The uniform distribution of \(\alpha_k\) is not essential. The same multiscale argument applies to any radius distribution \(\mu\) satisfying
\[
\mu([a,b])>0,
\qquad
\forall\,0<a<b<\bar\alpha.
\]
What matters is that the method can sample arbitrarily fine perturbation scales with positive probability. Neither a relative-rate condition such as \eqref{eq:rate-radii-ratio} nor a cluster-point-dependent upper bound on \(\bar\alpha\) is required.
\end{remark}

We now combine the descent property with the multiscale active-piece selection result to establish d-stationarity of every cluster point in the following theorem.

\begin{theorem}
\label{thm:hybrid-pdca-dstationarity}
Suppose that Assumptions~\ref{ass:coercive} and \ref{ass:linear_independence} hold. Let \(\{x^k\}\) be generated by Algorithm~\ref{alg:hybrid-pdca}. Then, with probability one, every cluster point of \(\{x^k\}\) is a d-stationary point of \(\zeta\).
\end{theorem}

\begin{proof}
By Lemma~\ref{lem:hybrid-descent-boundedness}, \(\{x^k\}\) is bounded. Fix a cluster point \(x^*\), and let \(\{x^k\}_{k\in\mathcal K}\) be a subsequence satisfying
\[
x^k\to x^*,
\qquad k\in\mathcal K\to\infty.
\]
Fix any \(i\in\mathcal M(x^*)\). By Lemma~\ref{lem:hybrid-active-selection}, almost surely there exists a subsequence \(\mathcal K_i\subseteq\mathcal K\) such that
\[
\alpha_k\to0,\qquad
\mathcal M(\hat x^k)=\{i\},\qquad
\hat x^k\to x^*,
\quad k\in\mathcal K_i.
\]
Along this subsequence,
\[
g^k=\nabla\psi_i(\hat x^k).
\]

We next show that the candidate residual \(v_k=\|y^k-\hat x^k\|\) vanishes along \(\mathcal K_i\). If the candidates $y^k$ are accepted infinitely many times, then \(v_k=r_k\to0\) by Lemma~\ref{lem:hybrid-descent-boundedness}. 

Suppose instead that the candidates $y^k$ are accepted only finitely many times; then they must be rejected infinitely many times. When $y^k$ is rejected, we have
\begin{equation}
\label{eq:hybrid-rejection}
\zeta(y^k)>\zeta(x^k)-\eta v_k^2.
\end{equation}
Define the Bregman remainder of \(\psi_i\) by
\[
D_i(x,z)
:=
\psi_i(x)-\psi_i(z)
-\langle\nabla\psi_i(z),x-z\rangle
\ge0.
\]
Since
\[
y^k=T\bigl(\hat x^k,\nabla\psi_i(\hat x^k)\bigr),
\]
the strong convexity of the candidate subproblem, with \(x^k\) as the comparator, yields
\begin{equation}
\label{eq:hybrid-candidate-phi-bound}
\begin{aligned}
\phi(y^k)
\le{}&
\phi(x^k)
+\langle g^k,y^k-x^k\rangle
+\frac{\sigma}{2}\|x^k-\hat x^k\|^2
-\frac{\sigma}{2}\|y^k-\hat x^k\|^2
-\frac{\sigma}{2}\|y^k-x^k\|^2.
\end{aligned}
\end{equation}
Moreover, since \(i\in\mathcal M(\hat x^k)\), the convexity of \(\psi_i\) gives
\begin{equation}
\label{eq:hybrid-candidate-psi-bound}
-\psi(y^k)
\le
-\psi_i(\hat x^k)
-\langle g^k,y^k-\hat x^k\rangle.
\end{equation}
Combining \eqref{eq:hybrid-candidate-phi-bound} and \eqref{eq:hybrid-candidate-psi-bound}, and using \(\|x^k-\hat x^k\|=\alpha_k\), we obtain
\begin{equation}
\label{eq:hybrid-candidate-upper}
\begin{aligned}
\zeta(y^k)
\le{}&
\zeta(x^k)
+\bigl[\psi(x^k)-\psi_i(x^k)\bigr]
+D_i(x^k,\hat x^k)
+\frac{\sigma}{2}\alpha_k^2
-\frac{\sigma}{2}v_k^2
-\frac{\sigma}{2}\|y^k-x^k\|^2.
\end{aligned}
\end{equation}
Combining \eqref{eq:hybrid-rejection} and \eqref{eq:hybrid-candidate-upper} gives
\[
\begin{aligned}
\left(\frac{\sigma}{2}-\eta\right)v_k^2
+\frac{\sigma}{2}\|y^k-x^k\|^2
<
&\ \psi(x^k)-\psi_i(x^k)
+D_i(x^k,\hat x^k)
+\frac{\sigma}{2}\alpha_k^2.
\end{aligned}
\]
Along \(\mathcal K_i\),
\[
x^k\to x^*,\qquad
\alpha_k\to0,\qquad
\hat x^k\to x^*.
\]
Since \(i\in\mathcal M(x^*)\),
\[
\psi(x^k)-\psi_i(x^k)
\to
\psi(x^*)-\psi_i(x^*)=0.
\]
The continuous differentiability of \(\psi_i\) also gives
\[
D_i(x^k,\hat x^k)\to0.
\]
Because \(\eta<\sigma/2\), it follows that \(v_k\to0\) along the rejected iterations in \(\mathcal K_i\). Therefore, whether the candidate is accepted or rejected,
\[
y^k-\hat x^k\to0,
\qquad k\in\mathcal K_i\to\infty.
\]
Since \(\hat x^k\to x^*\), we also have \(y^k\to x^*\).

Finally, the continuity of \(\nabla\psi_i\) implies
\[
g^k=\nabla\psi_i(\hat x^k)\to\nabla\psi_i(x^*).
\]
For any \(x\in\mb R^n\), the optimality of \(y^k=T(\hat x^k,g^k)\) gives
\begin{equation}
\label{eq:hybrid-subproblem-ineq}
\begin{aligned}
&\phi(y^k)
-\langle g^k,y^k-\hat x^k\rangle
+\frac{\sigma}{2}\|y^k-\hat x^k\|^2\\
\le&
\phi(x)
-\langle g^k,x-\hat x^k\rangle
+\frac{\sigma}{2}\|x-\hat x^k\|^2.
\end{aligned}
\end{equation}
Letting \(k\to\infty\) along \(\mathcal K_i\) in \eqref{eq:hybrid-subproblem-ineq}, and using the lower semicontinuity of \(\phi\), the continuity of \(\nabla\psi_i\), and \(y^k,\hat x^k\to x^*\), we obtain
\[
\phi(x^*)
\le
\phi(x)
-\langle\nabla\psi_i(x^*),x-x^*\rangle
+\frac{\sigma}{2}\|x-x^*\|^2,
\qquad
\forall x\in\mb R^n.
\]
Since \(i\in\mathcal M(x^*)\) was arbitrary, this relation holds for every active index. Therefore, by Proposition~\ref{pro:d-stationary}, \(x^*\) is a d-stationary point of \(\zeta\).
\end{proof}

\begin{remark}
The convergence analysis above does not use the global smoothness condition in Assumption~\ref{ass:psi}. 
In the rejected-candidate estimate, continuous differentiability is sufficient to obtain \(D_i(x^k,\hat x^k)\to0\).
Thus, the theorem requires only Assumptions~\ref{ass:coercive} and \ref{ass:linear_independence}, together with the randomized sampling rule in Algorithm~\ref{alg:hybrid-pdca} and the standing assumption that each \(\psi_i\) is continuously differentiable.
\end{remark}

\begin{remark}
The hybrid construction naturally raises the question of whether a randomized active-index method can similarly be combined with a standard DCA step. During the revision of this manuscript, we found a very recently proposed Hybrid Random Index-DCA~\cite[Alg.~2]{lethi2026finding} follows this idea: it samples one index from \(\mathcal M_\epsilon(x^k)\) and one from \(\mathcal M(x^k)\), solves the two associated subproblems, and selects the candidate with the smaller objective value. However, it still requires constructing the relaxed active index set and solving two subproblems at every iteration. In contrast, hybrid perturbed DCA selects an active gradient at a randomly perturbed point without explicitly constructing \(\mathcal M_\epsilon(x^k)\), and solves a fallback subproblem only when necessary. It therefore avoids the cost and potentially inefficient sampling associated with a large $\mc M_\epsilon(x^k)$. Moreover, perturbed DCA-type methods can be directly extended to apply to a general nonsmooth convex function \(\psi\) due to Rademacher's theorem, whereas random index selection cannot identify a prescribed index with positive probability when the index set is infinite. Hybrid Random Index-DCA~\cite[Alg.~2]{lethi2026finding} is included as a benchmark in our numerical experiments in the next section.
\end{remark}
}

\section{Numerical Experiments}\label{sec:numerical}

In this section, we show some numerical experiment results to demonstrate the effectiveness of the proposed methods for computing d-stationary points of the nonsmooth DC program \eqref{eq:dc_general}. All the numerical experiment results in this section are obtained by running MATLAB (Version R2025b) on a MacBook Air equipped with an Apple M4 chip and 16 GB of memory, under the macOS Sequoia 15.5 operating system. 
Next, we introduce the {benchmark methods and} the stopping criteria of the methods.

\paragraph{Benchmark methods} {We compare the following algorithms in the numerical experiments: perturbed DCA (Algorithm~\ref{alg:pDCA}), hybrid perturbed DCA (Algorithm~\ref{alg:hybrid-pdca}), DCA (Algorithm~\ref{alg:DCA}~\cite{tao1997convex,an2005dc}), revised DCA (Algorithm~\ref{alg:Pang}~\cite[Alg.~1]{pang2017computing}), revised DCA-Rand~\cite[Sec.~5.2]{pang2017computing}, and Hybrid Random Index-DCA~\cite[Alg.~2]{lethi2026finding}. 
}

\paragraph{Stopping criteria} Without loss of generality, we write $\phi$ in \eqref{eq:dc_general} as $\phi = \phi_1 + \phi_2$, where $\phi_1: \mathbb{R}^n \to (-\infty, +\infty]$ is a proper closed convex function whose proximal mapping is computable, and $\phi_2: \mathcal{X} \subseteq \mathbb{R}^n \to \mathbb{R}$ is a continuously differentiable convex function on the open set $\mathcal{X}$ containing ${\rm dom}(\phi_1)$. For any given $x \in {\rm dom}(\phi)$, we define:
\begin{align}\label{eq:d-sta_error}
\mathcal{R}(x) := \max_{i \in \mc M(x)} \frac{\left\|x - \text{prox}_{\phi_1}\big(x - (\nabla\phi_2(x) - \nabla\psi_i(x))\big)\right\|}{1+\|x\|+\|\nabla\phi_2(x)\|+ \|\nabla\psi_i(x)\|}.
\end{align}
It follows from Proposition~\ref{pro:d-stationary} that $x^* \in {\rm dom}(\phi)$ is a d-stationary point of \eqref{eq:dc_general} if and only if $\mathcal{R}(x^*) = 0$. Therefore, we terminate an algorithm when
\begin{equation*}
\mathcal{R}(x^k) < \tau \quad \mbox{or} \quad k > 10^5 
\end{equation*}
for some predefined $\tau > 0$. Since computing $\mathcal{R}(\cdot)$ in \eqref{eq:d-sta_error} can be time-consuming, in practice, we only check it when
$$
\dfrac{\|{x}^{k+1} - {x}^k\|}{\max\left\{1, \|{x}^{k+1}\|\right\}} < \tau \quad\text{and}\quad k < 10^5.
$$
{In all experiments of this section, we use the stopping accuracy $\tau=10^{-8}$. Each algorithm is run for 10 independent trials on each parameter setting or dataset. The time limit of 300 seconds is also imposed for each algorithm in each trial. If an algorithm does not reach the stopping criteria within the time limit, we terminate it and report the best solution found so far.}

\subsection{A simple comparison with DCA}

In this subsection, we provide a simple example to partially demonstrate that while both DCA and perturbed DCA solve a single subproblem for a one-step update, perturbed DCA can find a d-stationary point, but DCA is trapped at a critical point.

Consider the following example adopted from \cite{pang2017computing}:
\begin{equation}
\label{ex: pdcavsdca}
\min_{x\in\mathbb{R}}\quad \zeta(x):=\frac{1}{2}x^2 - \max\{-x,0\},
\end{equation}
which is an instance of \eqref{eq:dc_general} by taking $\phi(x) = \frac{1}{2}x^2$, $\psi_1(x) = -x$, $\psi_2(x) = 0$, and $\psi(x) = \max_{i \in\{1,2\}}\{\psi_i(x)\}$. The unique d-stationary point of the above problem is $x^* = -1$ with $\zeta(x^*) = -0.5$.

We compare the perturbed DCA (Algorithm~\ref{alg:pDCA} with $\sigma = 1$) with the standard DCA (Algorithm \ref{alg:DCA} by randomly selecting an index from $\mc M(x^k)$ at each iteration $k$ to compute $g^k \in \partial\psi(x^k)$ \cite[Ex.~8.31]{rw1998}), both starting from $x_0 = 1.5$. The performance of the two algorithms is shown in Fig.~\ref{fig:compare_with_dca}. From the results, we see that the sequence generated by DCA is trapped at the critical point $x = 0$. In contrast, benefiting from the random perturbations, the sequence generated by the perturbed DCA successfully finds the d-stationary point $x = -1$.

\begin{figure}[htbp]
    \centering
    \begin{subfigure}[b]{0.4\textwidth}
        \centering
        \includegraphics[width=\linewidth]{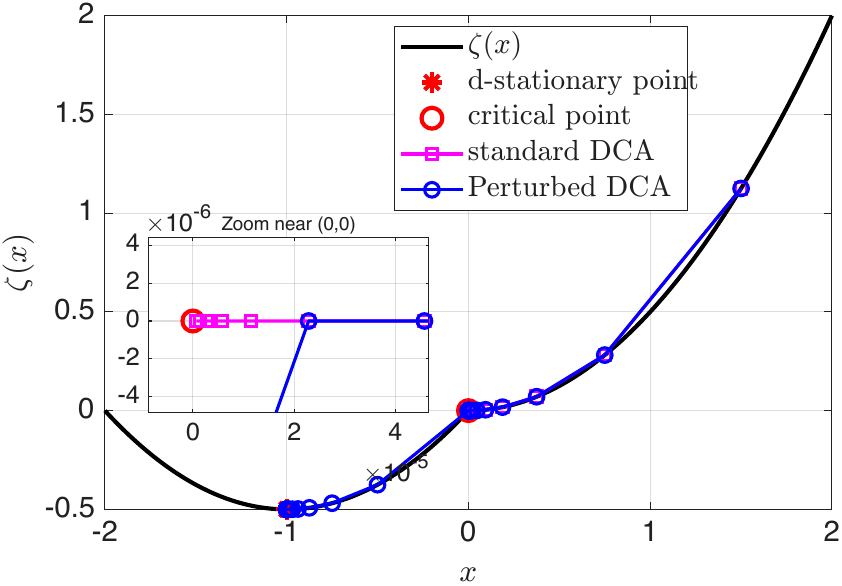}
        \caption{Convergence Path}
        \label{subfig:compare_with_dca_path}
    \end{subfigure}
    \hfill
    \begin{subfigure}[b]{0.4\textwidth}
        \centering
        \includegraphics[width=\linewidth]{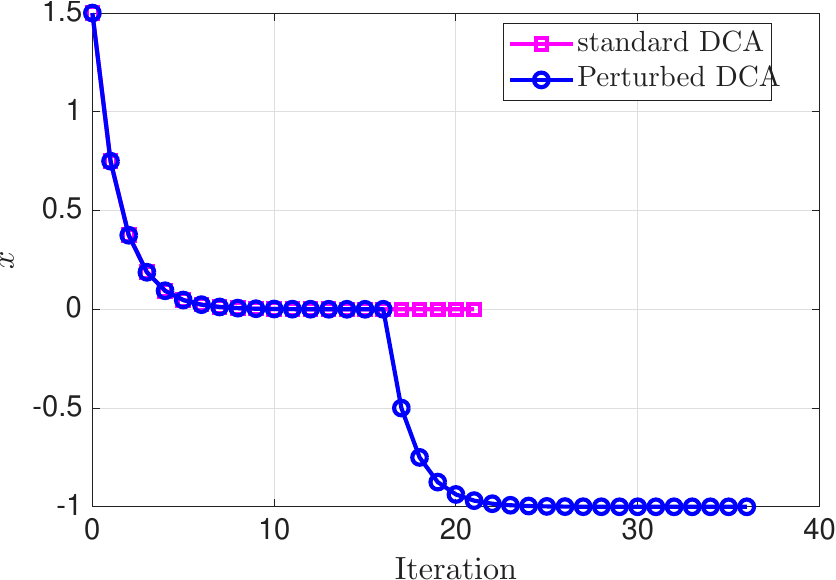}
        \caption{Variable}
        \label{subfig:compare_with_dca_x}
    \end{subfigure}
    
    \vspace{0.5cm}
    
    \begin{subfigure}[b]{0.4\textwidth}
        \centering
        \includegraphics[width=\linewidth]{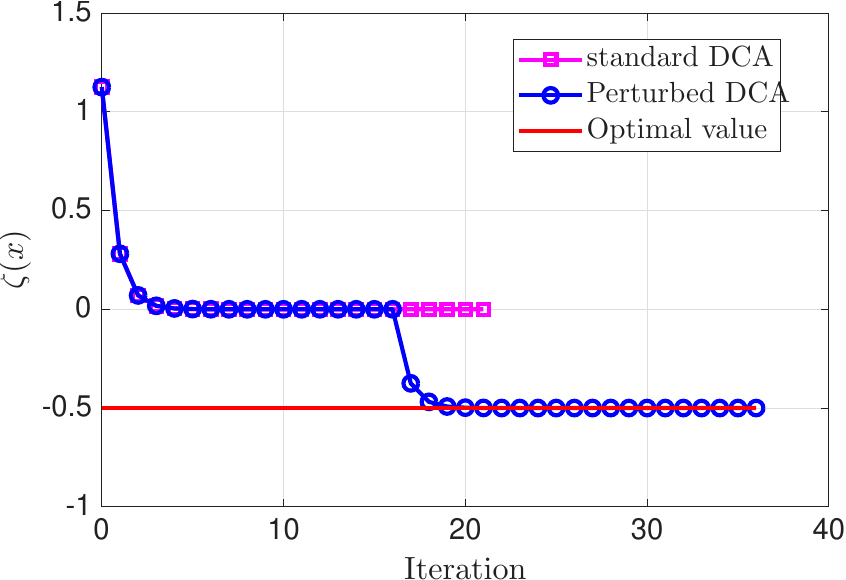}
        \caption{Function Value}
        \label{subfig:compare_with_dca_obj}
    \end{subfigure}
    \hfill
    \begin{subfigure}[b]{0.4\textwidth}
        \centering
        \includegraphics[width=\linewidth]{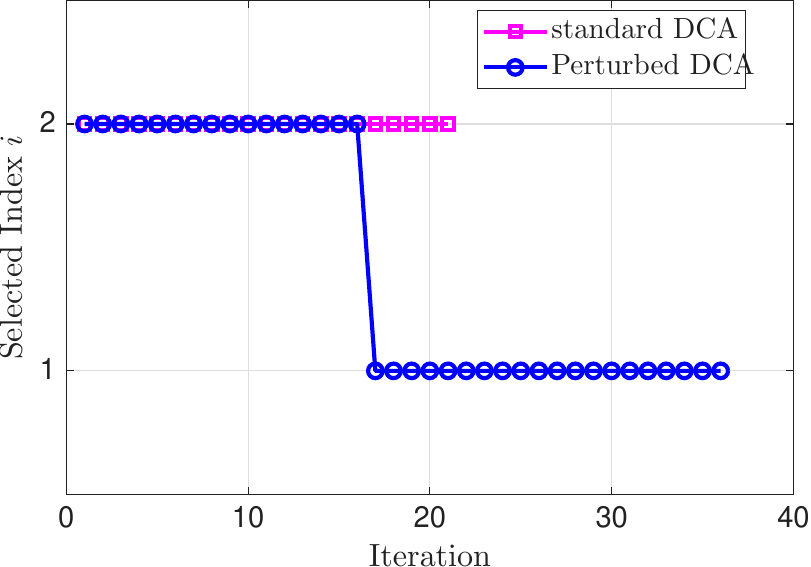}
        \caption{Selected Index}
        \label{subfig:compare_with_dca_index}
    \end{subfigure}
    
    \caption{Comparison of perturbed DCA and standard DCA for solving \eqref{ex: pdcavsdca}.}
    \label{fig:compare_with_dca}
\end{figure}

\subsection{Numerical results on $K$-sparse regularized linear regression problem}\label{subsec:num_Ksparse}

Recall the $K$-sparse regularized linear regression problem~\eqref{eq:dc_ksparse-lr}:
\begin{equation*}
    \min_{x\in\mb R^n}\quad \zeta(x):=\underbrace{\frac12\|Ax-b\|^2 + \lambda\|x\|_1}_{\phi(x)}-\underbrace{\lambda\|x\|_{(K)}}_{\psi(x)},
\end{equation*}
where the data $A \in \mathbb{R}^{m \times n}$, $b \in \mathbb{R}^m$, and the parameter $\lambda > 0$ are given.

\paragraph{Data generation}
For any given positive integers $m$, $n$, and $1 < K < n$, we generate the true sparse signal ${x}_{\text{true}} \in \mathbb{R}^n$ with $K$ non-zero entries at random positions, with values i.i.d. drawn from the standard normal distribution $\mathcal{N}(0, 1)$. Each entry of the measurement matrix ${A} \in \mathbb{R}^{m \times n}$ is i.i.d. drawn from the standard normal distribution $\mathcal{N}(0, 1)$, and then normalized so that the columns have unit length. The observations are obtained as ${b} = {A}{x}_{\text{true}} + {\epsilon}$, where ${\epsilon} \sim \mathcal{N}(0, \varepsilon^2I)$. We set $\varepsilon=0.1$ in our experiments.

\paragraph{Active piece selection}

At each iteration \(k\geq0\), after obtaining the perturbed point \(\hat x^k=x^k+e^k\), we select an affine piece of \(\lambda\|x\|_{(K)}\) from \(\hat x^k\). Let \(\mc I_k\) be the index set corresponding to the \(K\) largest absolute values of \(\hat x^k\), and define
\[
s^k=\sgn(\hat x^k), \qquad
(w^k)_i=
\begin{cases}
1, & i\in\mc I_k,\\
0, & \text{otherwise}.
\end{cases}
\]
Then the selected active gradient is
\[
g^k=\lambda(w^k\circ s^k).
\]
For perturbed DCA, this gives the unique element of \(\mathcal S(\hat x^k)\) almost surely. For hybrid perturbed DCA, the active gradient is likewise constructed at the perturbed point,
\[
g^k=\nabla\psi_{i_k}(\hat x^k),
\qquad
y^k=T(\hat x^k,g^k).
\]
Although the gradient of the selected affine piece is independent of its evaluation point, the proximal center of the hybrid candidate remains \(\hat x^k\). If the candidate is rejected, the fallback DCA step selects an active piece at \(x^k\) and uses \(x^k\) as the proximal center. Note that the computational complexity for constructing the active gradient $g^k$ is $\mc O(n\log n)$.

\subsubsection{A dual semismooth Newton method for solving the subproblems}
The subproblems of the perturbed DCA {and hybrid perturbed DCA} for solving the $K$-sparse regularized linear regression problem~\eqref{eq:dc_ksparse-lr} take the following form (up to some constants independent of $x$)
\begin{equation}
\label{eq: pdca_kparse_sub}
\min_{x \in \mathbb{R}^n} ~ \frac{1}{2}\|Ax\|^2 + \lambda\|x\|_1 + \langle c, x \rangle + \frac{\sigma}{2}\|x - \tilde{x}\|^2,
\end{equation}
where $c \in \mathbb{R}^n$ and $\tilde{x} \in \mathbb{R}^n$ are some given vectors. {For the perturbed DCA candidate and the hybrid perturbed DCA candidate, we set \(\tilde x=\hat x^k\). For the fallback step of hybrid perturbed DCA, we instead set \(\tilde x=x^k\).} The dual problem of \eqref{eq: pdca_kparse_sub} is given by
\begin{equation}
\label{eq: dual_pdca_ksparse}
\min_{z \in \mathbb{R}^m} ~ F(z) := \frac{1}{2}\|z\|^2 - \sigma M_{\frac{\lambda}{\sigma}\|\cdot\|_1}\left(\tilde{x} - \frac{c + A^\top z}{\sigma}\right) - \left(\langle \tilde{x}, c + A^\top z \rangle - \frac{1}{2\sigma}\|c + A^\top z\|^2\right),
\end{equation}
where $M_{\lambda/\sigma\|\cdot\|_1}(\cdot)$ is the Moreau envelope of the convex function $\lambda/\sigma\|\cdot\|_1$. 

Many efficient algorithms can be applied to the convex optimization problems \eqref{eq: pdca_kparse_sub} and \eqref{eq: dual_pdca_ksparse}, such as accelerated first-order methods \cite{nesterov1983method,beck2009fast,sun2025accelerating,zhang2022efficient}. By realizing the dual problem \eqref{eq: dual_pdca_ksparse} admits a unique solution $z^*$ as the solution to the following nonsmooth equation 
\begin{equation}
\label{eq: neq-dualksparse}
\nabla F(z) := z - A\prox_{\frac{\lambda}{\sigma}\|\cdot\|_1}\left(\tilde{x} - \frac{1}{\sigma}\left(c + A^\top z\right)\right) = 0,
\end{equation}
and the solution $x^*$ to \eqref{eq: pdca_kparse_sub} is given by
\begin{align*}
x^* = \prox_{\frac{\lambda}{\sigma}\|\cdot\|_1}\left(\tilde{x} - \frac{1}{\sigma}\left(c + A^\top z^*\right)\right),
\end{align*}
we will apply the globalized semismooth Newton method \cite{li2018highly,facchinei2007finite, lin2024highly} to find the solution $z^*$ to \eqref{eq: neq-dualksparse}, which can exploit the so-called ``second-order sparsity'' of the problem \cite{li2018highly}. Moreover, the solution to \eqref{eq: pdca_kparse_sub} is usually sparse, so we can apply the dimension reduction technique, adaptive sieving, developed in \cite{yuan2025adaptive,yuan2022dimension} to further reduce the computation time.

\subsubsection{Numerical results}

We test the compared algorithms for solving the $K$-sparse regularized linear regression problem under different settings of $(m,n,K)$ and the parameter $\lambda$. The numerical results are reported in Tables~\ref{tab:Ksparse-lambda_0.1} and \ref{tab:Ksparse-lambda_0.05}. {Overall, perturbed DCA and hybrid perturbed DCA are much more efficient than revised DCA and revised DCA-Rand. Revised DCA solves all subproblems associated with the selected $\mc M_\epsilon$, and hence its total number of subproblems is substantially larger. Revised DCA-Rand reduces the number of subproblems by sampling a single index from $\mc M_\epsilon$, but its repeated candidate rejections can increase the iteration number. As a result, both revised DCA and revised DCA-Rand require much longer computation time than the perturbed DCA-type methods on these large-scale sparse regression instances. Hybrid Random Index-DCA is faster than revised DCA and revised DCA-Rand, but it is still slower than perturbed DCA and hybrid perturbed DCA as it solves two subproblems at each iteration.}

\begin{table}[t]
    \centering
    \begin{minipage}{\dimexpr\textwidth-2\fboxsep\relax}
    \caption{Experimental results on $K$-sparse regularized linear regression problems~\eqref{eq:dc_ksparse-lr} with $\lambda = 1.00 \times 10^{-1}$. Each parameter set $(m,n,K)$ is averaged over 10 trials, and the reported values are ``mean $\pm$ standard deviation''. ``Iter.'' denotes the number of iterations, ``Time'' is the CPU time in seconds (s), ``Subprob.'' denotes the number of subproblems solved, and ``Reject'' denotes the number of rejected candidates.}
    \label{tab:Ksparse-lambda_0.1}
    \resizebox{\textwidth}{!}{%
    \begin{tabular}{c c c l llllll}
        \toprule
        \multirow{3}{*}{\textbf{$m$}} & \multirow{3}{*}{\textbf{$n$}} & \multirow{3}{*}{\textbf{$K$}} & & \multirow{3}{*}{\shortstack[l]{\textbf{Perturbed}\\\textbf{DCA}}} & \multirow{3}{*}{\shortstack[l]{\textbf{Hybrid}\\\textbf{Perturbed}\\\textbf{DCA}}} & \multirow{3}{*}{\textbf{DCA}} & \multirow{3}{*}{\shortstack[l]{\textbf{Revised}\\\textbf{DCA}}} & \multirow{3}{*}{\shortstack[l]{\textbf{Revised}\\\textbf{DCA-Rand}}} & \multirow{3}{*}{\shortstack[l]{\textbf{Hybrid}\\\textbf{Random Index}\\\textbf{DCA}}} \\
        &&&&&&&&& \\ 
        &&&&&&&&& \\ 
        \midrule
        \multirow{4}{*}{500} & \multirow{4}{*}{1000} & \multirow{4}{*}{20} 
        & Iter.    & $11 \pm 2$ & $\mathbf{8} \pm 1$ & $9 \pm 2$ & $\mathbf{8} \pm 5$ & $10 \pm 5$ & $\mathbf{8} \pm 1$ \\
        &&& Subprob. & $11 \pm 2$ & $10 \pm 1$ & $\mathbf{9} \pm 2$ & $129 \pm 42$ & $10 \pm 5$ & $16 \pm 2$ \\
        &&& Time     & $0.07 \pm 0.01$ & $\mathbf{0.05} \pm 0.01$ & $\mathbf{0.05} \pm 0.01$ & $1.92 \pm 0.10$ & $0.85 \pm 0.45$ & $0.66 \pm 0.10$ \\
        &&& Reject   & — & $\mathbf{2} \pm 1$ & — & — & $\mathbf{2} \pm 5$ & — \\
        \cmidrule(lr){1-10}
        
        \multirow{4}{*}{500} & \multirow{4}{*}{1000} & \multirow{4}{*}{50} 
        & Iter.    & $13 \pm 3$ & $\mathbf{9} \pm 1$ & $13 \pm 5$ & $10 \pm 1$ & $12 \pm 3$ & $\mathbf{9} \pm 1$ \\
        &&& Subprob. & $13 \pm 3$ & $\mathbf{12} \pm 2$ & $13 \pm 5$ & $174 \pm 83$ & $\mathbf{12} \pm 3$ & $19 \pm 2$ \\
        &&& Time     & $0.09 \pm 0.02$ & $\mathbf{0.07} \pm 0.01$ & $0.08 \pm 0.02$ & $2.11 \pm 0.13$ & $1.02 \pm 0.30$ & $0.82 \pm 0.13$ \\
        &&& Reject   & — & $\mathbf{4} \pm 1$ & — & — & $\mathbf{4} \pm 2$ & — \\
        \cmidrule(lr){1-10}
        
        \multirow{4}{*}{500} & \multirow{4}{*}{1000} & \multirow{4}{*}{100} 
        & Iter.    & $14 \pm 5$ & $11 \pm 1$ & $15 \pm 5$ & $\mathbf{10} \pm 1$ & $17 \pm 8$ & $11 \pm 1$ \\
        &&& Subprob. & $\mathbf{14} \pm 5$ & $\mathbf{14} \pm 2$ & $15 \pm 5$ & $198 \pm 80$ & $17 \pm 8$ & $21 \pm 2$ \\
        &&& Time     & $\mathbf{0.11} \pm 0.02$ & $\mathbf{0.11} \pm 0.03$ & $0.13 \pm 0.02$ & $2.34 \pm 0.16$ & $1.55 \pm 0.80$ & $1.05 \pm 0.16$ \\
        &&& Reject   & — & $\mathbf{3} \pm 1$ & — & — & $6 \pm 7$ & — \\
        \cmidrule(lr){1-10}
        
        \multirow{4}{*}{1000} & \multirow{4}{*}{2000} & \multirow{4}{*}{100} 
        & Iter.    & $15 \pm 6$ & $\mathbf{9} \pm 2$ & $12 \pm 4$ & $10 \pm 1$ & $12 \pm 2$ & $10 \pm 2$ \\
        &&& Subprob. & $15 \pm 6$ & $13 \pm 3$ & $\mathbf{12} \pm 4$ & $143 \pm 42$ & $\mathbf{12} \pm 2$ & $20 \pm 4$ \\
        &&& Time     & $0.19 \pm 0.04$ & $\mathbf{0.15} \pm 0.02$ & $0.16 \pm 0.04$ & $3.76 \pm 0.15$ & $1.86 \pm 0.45$ & $1.58 \pm 0.29$ \\
        &&& Reject   & — & $4 \pm 1$ & — & — & $\mathbf{3} \pm 2$ & — \\
        \cmidrule(lr){1-10}
        
        \multirow{4}{*}{2000} & \multirow{4}{*}{4000} & \multirow{4}{*}{200} 
        & Iter.    & $18 \pm 6$ & $\mathbf{10} \pm 2$ & $17 \pm 7$ & $12 \pm 2$ & $13 \pm 3$ & $\mathbf{10} \pm 2$ \\
        &&& Subprob. & $18 \pm 6$ & $\mathbf{13} \pm 3$ & $17 \pm 7$ & $218 \pm 103$ & $\mathbf{13} \pm 3$ & $21 \pm 4$ \\
        &&& Time     & $0.88 \pm 0.25$ & $\mathbf{0.65} \pm 0.05$ & $0.82 \pm 0.30$ & $13.65 \pm 2.65$ & $3.71 \pm 0.96$ & $3.22 \pm 0.62$ \\
        &&& Reject   & — & $\mathbf{3} \pm 1$ & — & — & $4 \pm 2$ & — \\
        \cmidrule(lr){1-10}
        
        \multirow{4}{*}{5000} & \multirow{4}{*}{10000} & \multirow{4}{*}{500} 
        & Iter.    & $14 \pm 2$ & $\mathbf{11} \pm 1$ & $14 \pm 2$ & $17 \pm 5$ & $18 \pm 6$ & $\mathbf{11} \pm 3$ \\
        &&& Subprob. & $\mathbf{14} \pm 2$ & $\mathbf{14} \pm 2$ & $\mathbf{14} \pm 2$ & $259 \pm 132$ & $18 \pm 6$ & $23 \pm 5$ \\
        &&& Time     & $2.93 \pm 0.54$ & $\mathbf{2.84} \pm 0.27$ & $2.86 \pm 0.45$ & $43.87 \pm 16.91$ & $9.31 \pm 2.85$ & $8.19 \pm 1.65$ \\
        &&& Reject   & — & $\mathbf{3} \pm 1$ & — & — & $6 \pm 3$ & — \\
        \bottomrule
    \end{tabular}%
    }
    \end{minipage}
\end{table}

\begin{table}[t]
    \centering
    \begin{minipage}{\dimexpr\textwidth-2\fboxsep\relax}
    \caption{Experimental results on $K$-sparse regularized linear regression problems~\eqref{eq:dc_ksparse-lr} with $\lambda = 5.00 \times 10^{-2}$. Each parameter set $(m,n,K)$ is averaged over 10 trials, and the reported values are ``mean $\pm$ standard deviation''. ``Iter.'' denotes the number of iterations, ``Time'' is the CPU time in seconds (s), ``Subprob.'' denotes the number of subproblems solved, and ``Reject'' denotes the number of rejected candidates.}
    \label{tab:Ksparse-lambda_0.05}
    \resizebox{\textwidth}{!}{%
    \begin{tabular}{c c c l llllll}
        \toprule
        \multirow{3}{*}{\textbf{$m$}} & \multirow{3}{*}{\textbf{$n$}} & \multirow{3}{*}{\textbf{$K$}} & & \multirow{3}{*}{\shortstack[l]{\textbf{Perturbed}\\\textbf{DCA}}} & \multirow{3}{*}{\shortstack[l]{\textbf{Hybrid}\\\textbf{perturbed}\\\textbf{DCA}}} & \multirow{3}{*}{\textbf{DCA}} & \multirow{3}{*}{\shortstack[l]{\textbf{Revised}\\\textbf{DCA}}} & \multirow{3}{*}{\shortstack[l]{\textbf{Revised}\\\textbf{DCA-Rand}}} & \multirow{3}{*}{\shortstack[l]{\textbf{Hybrid}\\\textbf{Random Index}\\\textbf{DCA}}} \\
        &&&&&&&&& \\ 
        &&&&&&&&& \\ 
        \midrule
        \multirow{4}{*}{500} & \multirow{4}{*}{1000} & \multirow{4}{*}{20} 
        & Iter.    & $11 \pm 0$ & $\mathbf{8} \pm 0$ & $9 \pm 2$ & $\mathbf{8} \pm 1$ & $13 \pm 9$ & $\mathbf{8} \pm 2$ \\
        &&& Subprob. & $11 \pm 0$ & $11 \pm 1$ & $\mathbf{9} \pm 2$ & $136 \pm 47$ & $13 \pm 9$ & $17 \pm 4$ \\
        &&& Time     & $0.09 \pm 0.02$ & $\mathbf{0.06} \pm 0.03$ & $\mathbf{0.06} \pm 0.02$ & $2.61 \pm 0.23$ & $1.24 \pm 0.80$ & $0.83 \pm 0.26$ \\
        &&& Reject   & — & $\mathbf{3} \pm 1$ & — & — & $5 \pm 9$ & — \\
        \cmidrule(lr){1-10}
        
        \multirow{4}{*}{500} & \multirow{4}{*}{1000} & \multirow{4}{*}{50} 
        & Iter.    & $12 \pm 3$ & $9 \pm 2$ & $10 \pm 2$ & $10 \pm 1$ & $14 \pm 9$ & $\mathbf{8} \pm 1$ \\
        &&& Subprob. & $12 \pm 3$ & $11 \pm 3$ & $\mathbf{10} \pm 2$ & $249 \pm 41$ & $14 \pm 9$ & $17 \pm 1$ \\
        &&& Time     & $0.11 \pm 0.02$ & $\mathbf{0.09} \pm 0.03$ & $0.11 \pm 0.03$ & $3.31 \pm 0.31$ & $1.56 \pm 0.92$ & $0.94 \pm 0.10$ \\
        &&& Reject   & — & $\mathbf{3} \pm 1$ & — & — & $6 \pm 8$ & — \\
        \cmidrule(lr){1-10}
        
        \multirow{4}{*}{500} & \multirow{4}{*}{1000} & \multirow{4}{*}{100} 
        & Iter.    & $15 \pm 5$ & $\mathbf{9} \pm 1$ & $15 \pm 5$ & $12 \pm 3$ & $16 \pm 5$ & $11 \pm 2$ \\
        &&& Subprob. & $15 \pm 5$ & $\mathbf{12} \pm 1$ & $15 \pm 5$ & $166 \pm 66$ & $16 \pm 5$ & $21 \pm 3$ \\
        &&& Time     & $\mathbf{0.15} \pm 0.02$ & $0.16 \pm 0.08$ & $0.17 \pm 0.06$ & $3.40 \pm 0.42$ & $1.75 \pm 0.72$ & $1.27 \pm 0.17$ \\
        &&& Reject   & — & $\mathbf{3} \pm 1$ & — & — & $5 \pm 4$ & — \\
        \cmidrule(lr){1-10}
        
        \multirow{4}{*}{1000} & \multirow{4}{*}{2000} & \multirow{4}{*}{100} 
        & Iter.    & $13 \pm 4$ & $\mathbf{8} \pm 0$ & $10 \pm 3$ & $9 \pm 1$ & $17 \pm 6$ & $9 \pm 1$ \\
        &&& Subprob. & $13 \pm 4$ & $\mathbf{10} \pm 1$ & $\mathbf{10} \pm 3$ & $154 \pm 38$ & $17 \pm 6$ & $18 \pm 3$ \\
        &&& Time     & $\mathbf{0.21} \pm 0.02$ & $\mathbf{0.21} \pm 0.02$ & $0.22 \pm 0.04$ & $5.44 \pm 0.38$ & $3.85 \pm 1.39$ & $2.08 \pm 0.34$ \\
        &&& Reject   & — & $\mathbf{2} \pm 1$ & — & — & $8 \pm 5$ & — \\
        \cmidrule(lr){1-10}
        
        \multirow{4}{*}{2000} & \multirow{4}{*}{4000} & \multirow{4}{*}{200} 
        & Iter.    & $14 \pm 5$ & $\mathbf{9} \pm 2$ & $12 \pm 4$ & $11 \pm 3$ & $16 \pm 8$ & $10 \pm 1$ \\
        &&& Subprob. & $14 \pm 5$ & $\mathbf{12} \pm 3$ & $\mathbf{12} \pm 4$ & $160 \pm 32$ & $16 \pm 8$ & $19 \pm 2$ \\
        &&& Time     & $\mathbf{0.81} \pm 0.18$ & $0.82 \pm 0.12$ & $0.85 \pm 0.20$ & $14.39 \pm 2.08$ & $4.86 \pm 2.13$ & $3.44 \pm 0.39$ \\
        &&& Reject   & — & $\mathbf{3} \pm 1$ & — & — & $7 \pm 7$ & — \\
        \cmidrule(lr){1-10}
        
        \multirow{4}{*}{5000} & \multirow{4}{*}{10000} & \multirow{4}{*}{500} 
        & Iter.    & $12 \pm 1$ & $\mathbf{10} \pm 2$ & $\mathbf{10} \pm 2$ & $\mathbf{10} \pm 2$ & $18 \pm 7$ & $\mathbf{10} \pm 2$ \\
        &&& Subprob. & $12 \pm 1$ & $12 \pm 3$ & $\mathbf{10} \pm 2$ & $233 \pm 137$ & $18 \pm 7$ & $19 \pm 5$ \\
        &&& Time     & $2.87 \pm 0.16$ & $2.90 \pm 0.40$ & $\mathbf{2.48} \pm 0.24$ & $28.51 \pm 2.96$ & $11.48 \pm 4.16$ & $7.69 \pm 1.33$ \\
        &&& Reject   & — & $\mathbf{3} \pm 1$ & — & — & $9 \pm 7$ & — \\
        \bottomrule
    \end{tabular}%
    }
    \end{minipage}
\end{table}

\subsection{Numerical results on $K$-medians clustering problem}

In this subsection we consider the problem~\eqref{eq:dc_clustering} with $p=1$:
\begin{equation*}
\min_{\mu \in \mathbb{R}^{d \times K}}~\zeta(\mu):= \underbrace{\frac{1}{n}\sum_{i=1}^n \sum_{l=1}^K \|\mu_l - a_i\|_1}_{\phi(\mu)} - \underbrace{\frac{1}{n}\sum_{i = 1}^n \max_{1 \leq j \leq K} \left(\sum_{1 \leq l \neq j \leq K} \|\mu_l - a_i\|_1\right)}_{\psi(\mu)},
\end{equation*}
where the dataset $\{a_1,\ldots,a_n\}\subseteq\mb R^d$ is given.

The experiments are conducted on real-world datasets \texttt{Iris}, \texttt{Wine}, \texttt{Glass}, and \texttt{Yeast} from the UCI Machine Learning Repository\footnote{https://archive.ics.uci.edu/}. 
We compute an initial candidate solution $\mu_{\text{kmed}}$ using the MATLAB function \texttt{kmedoids} with 5 replicates, {and use the same $\mu_{\text{kmed}}$ as the starting point for all compared algorithms.}

\subsubsection{The solution of subproblems}

Both Algorithm~\ref{alg:pDCA} {and Algorithm~\ref{alg:hybrid-pdca}} lead to subproblems of the following unified form. Given a proximal center \(Z^k\in\mathbb R^{d\times K}\) and a selected active (sub)gradient \(P^k\in\mathbb R^{d\times K}\), we compute
\begin{equation}
\label{eq:kmedians_subproblem}
\widetilde{\mu}^{k+1}
=
\underset{\mu \in \mathbb{R}^{d \times K}}{\arg\min}
\left\{
\phi(\mu)
-
\langle P^k,\mu-Z^k\rangle
+
\frac{\sigma}{2}\|\mu-Z^k\|_F^2
\right\}.
\end{equation}
Expanding the terms and ignoring constant factors independent of \(\mu\) in~\eqref{eq:kmedians_subproblem}, we obtain
\begin{equation}\label{eq:kmedians_expanded_obj}
    \min_{\mu \in \mathbb{R}^{d \times K}}~ \frac{1}{n} \sum_{i=1}^{n} \sum_{l=1}^{K} \|\mu_l-a_i\|_1 - \langle P^k,\mu\rangle + \frac{\sigma}{2}\|\mu\|_F^2 - \sigma\langle \mu,Z^k\rangle.
\end{equation}
Noting that both the \(\ell_1\)-norm and Frobenius norm are separable across cluster centers and coordinates, we can decompose \eqref{eq:kmedians_expanded_obj} into \(K\times d\) independent one-dimensional optimization problems. For each cluster center \(\mu_l\), \(l=1,\ldots,K\), and each coordinate \(r=1,\ldots,d\), we solve
\begin{equation}
\label{eq:kmedians_1d_subproblem}
    \min_{\mu_l^{(r)}\in\mathbb R}~\frac{1}{n}\sum_{i=1}^{n}|\mu_l^{(r)}-a_i^{(r)}| + \frac{\sigma}{2}(\mu_l^{(r)})^2 - c_l^{(r)}\mu_l^{(r)},
\end{equation}
where
\[
c_l^{(r)} = P_l^{k,(r)}+\sigma Z_l^{k,(r)}
\]
with $P_l^{k,(r)}$ and $Z_l^{k,(r)}$ being the $(r,l)$-entry of $P^k$ and $Z^k$, respectively.

\paragraph{Solution of the one-dimensional subproblem}
The subproblem \eqref{eq:kmedians_1d_subproblem} is a strictly convex piecewise quadratic function of a one-dimensional variable $x \triangleq \mu_l^{(r)}\in\mb R$. Its global minimizer can be found efficiently and exactly. For notational simplicity, we denote the fixed data points for this subproblem as $b_i \triangleq a_i^{(r)}$ for $i=1,\ldots,n$, and the constant $c \triangleq c_l^{(r)}$. The objective in \eqref{eq:kmedians_1d_subproblem} is then
\begin{equation}
f(x) = \frac{1}{n}\sum_{i=1}^{n} |x - b_i| + \frac{\sigma}{2} x^2 - c x.
\label{eq:1d_objective}
\end{equation}
A standard approach to solve \eqref{eq:1d_objective} exploits its piecewise linear derivative structure. Let $b_{(1)} \le b_{(2)} \le \cdots \le b_{(n)}$ be the sorted values of $\{b_i\}_{i=1}^n$. For any $x$ not equal to any $b_i$, the derivative of $f(x)$ is
\begin{equation*}
f'(x) = \frac{1}{n} \sum_{i=1}^{n} \sgn(x - b_i) + \sigma x - c.
\end{equation*}
In the open interval $(b_{(m)}, b_{(m+1)})$ between two consecutive sorted points, the sum of signs is constant: $\sum_{i=1}^{n} \sgn(x - b_i) = 2m - n$. Setting $f'(x)=0$ within this interval yields a candidate solution
\begin{equation}
x_m = \frac{c - (2m - n)/n}{\sigma}, \quad \text{provided } b_{(m)} < x_m < b_{(m+1)}.
\label{eq:candidate_solution}
\end{equation}
Since $f(x)$ is strictly convex, its unique global minimizer $x^*$ must either be one of these interval-stationary points $x_m$ satisfying the condition in \eqref{eq:candidate_solution}, or coincide with one of the data points $b_i$ where $f(x)$ is non-differentiable. Therefore, the exact solution can be obtained by the following procedure:
\begin{enumerate}
    \item[Step 1.] Sort the data points $\{b_i\}_{i=1}^n$ to obtain $b_{(1)} \le \cdots \le b_{(n)}$.
    \item[Step 2.] For each $m = 0, 1, \ldots, n$ (defining $b_{(0)} = -\infty$ and $b_{(n+1)} = \infty$), compute the candidate $x_m$ using \eqref{eq:candidate_solution} and check if $b_{(m)} < x_m < b_{(m+1)}$.
    \item[Step 3.] Evaluate the objective value $f(x)$ at all valid candidates $x_m$ from step 2 and at all data points $b_i$.
    \item[Step 4.] The point achieving the smallest value of $f(x)$ is the global minimizer $x^*$.
\end{enumerate}

\begin{table}[p] 
    \centering
    \rotatebox[origin=c]{90}{%
    \begin{minipage}{\dimexpr\textheight-2\fboxsep\relax} 
    \vspace{0.5em}
    \caption{Experimental results of $K$-medians clustering problems on four real-world datasets. Each algorithm was run for 10 independent trials on each dataset, and the reported values are ``mean $\pm$ standard deviation''. ``Succ.'' denotes the number of successful trials, i.e., achieving $\mc R(x^k)<\tau$ at termination, ``Iter.'' denotes the number of iterations, ``Obj.'' denotes the final objective function value $\zeta$, ``Subprob.'' denotes the number of subproblems, and ``Time'' denotes the CPU time in seconds.}
    \label{tab:kmedians}    
    \resizebox{\linewidth}{!}{%
    \begin{tabular}{l l llllll}
        \toprule
        \multirow{2}{*}{\shortstack[l]{\textbf{Dataset}\\$(n,d,K)$}} & & \multirow{2}{*}{\shortstack[l]{\textbf{Perturbed}\\\textbf{DCA}}} & \multirow{2}{*}{\shortstack[l]{\textbf{Hybrid} \\ \textbf{Perturbed DCA}}} & \multirow{2}{*}{\textbf{DCA}} & \multirow{2}{*}{\shortstack[l]{\textbf{Revised} \\ \textbf{DCA}}} & \multirow{2}{*}{\shortstack[l]{\textbf{Revised} \\ \textbf{DCA-Rand}}} & \multirow{2}{*}{\shortstack[l]{\textbf{Hybrid}\\\textbf{Random Index-DCA}}} \\
        &&&&&&& \\ 
        \midrule
        \multirow{5}{*}{\shortstack[l]{\texttt{Iris}\\$(150,4,3)$}} 
        & Succ.    & \textbf{10}/10 & \textbf{10}/10 & \textbf{10}/10 & \textbf{10}/10 & \textbf{10}/10 & \textbf{10}/10 \\
        & Iter.    & $5.8\pm2.1$ & $4.9\pm0.9$ & $4.2\pm0.6$ & $\textbf{4.1}\pm0.3$ & $4.9\pm2.8$ & $\textbf{4.1}\pm0.3$ \\
        & Obj.     & $\mathbf{1.065\mathrm{e}{+00}}\pm8.767\mathrm{e}{-03}$ & $\mathbf{1.065\mathrm{e}{+00}}\pm9.404\mathrm{e}{-03}$ & $\mathbf{1.065\mathrm{e}{+00}}\pm8.644\mathrm{e}{-03}$ & $\mathbf{1.065\mathrm{e}{+00}}\pm8.644\mathrm{e}{-03}$ & $\mathbf{1.065\mathrm{e}{+00}}\pm8.644\mathrm{e}{-03}$ & $\mathbf{1.065\mathrm{e}{+00}}\pm8.644\mathrm{e}{-03}$ \\
        & Subprob. & $5.8\pm2.1$ & $6.6\pm1.3$ & $\textbf{4.2}\pm0.6$ & $7.2\pm7.0$ & $4.9\pm2.8$ & $8.2\pm0.6$ \\
        & Time     & $\textbf{0.01}\pm0.00$ & $\textbf{0.01}\pm0.00$ & $\textbf{0.01}\pm0.00$ & $\textbf{0.01}\pm0.00$ & $\textbf{0.01}\pm0.01$ & $\textbf{0.01}\pm0.00$ \\
        \cmidrule(lr){1-8}
        
        \multirow{5}{*}{\shortstack[l]{\texttt{Wine}\\$(178,13,3)$}} 
        & Succ.    & \textbf{10}/10 & \textbf{10}/10 & \textbf{10}/10 & \textbf{10}/10 & \textbf{10}/10 & \textbf{10}/10 \\
        & Iter.    & $17.4\pm2.8$ & $17.8\pm2.1$ & $\textbf{16.0}\pm0.0$ & $\textbf{16.0}\pm0.0$ & $\textbf{16.0}\pm0.0$ & $\textbf{16.0}\pm0.0$ \\
        & Obj.     & $\mathbf{1.065\mathrm{e}{+02}}\pm2.096\mathrm{e}{-03}$ & $\mathbf{1.065\mathrm{e}{+02}}\pm4.264\mathrm{e}{-04}$ & $\mathbf{1.065\mathrm{e}{+02}}\pm1.872\mathrm{e}{-13}$ & $\mathbf{1.065\mathrm{e}{+02}}\pm1.872\mathrm{e}{-13}$ & $\mathbf{1.065\mathrm{e}{+02}}\pm1.872\mathrm{e}{-13}$ & $\mathbf{1.065\mathrm{e}{+02}}\pm1.872\mathrm{e}{-13}$ \\
        & Subprob. & $17.4\pm2.8$ & $19.9\pm2.4$ & $\textbf{16.0}\pm0.0$ & $\textbf{16.0}\pm0.0$ & $\textbf{16.0}\pm0.0$ & $32.0\pm0.0$ \\
        & Time     & $\textbf{0.02}\pm0.00$ & $\textbf{0.02}\pm0.00$ & $\textbf{0.02}\pm0.00$ & $\textbf{0.02}\pm0.00$ & $\textbf{0.02}\pm0.00$ & $0.03\pm0.00$ \\
        \cmidrule(lr){1-8}
        
        \multirow{5}{*}{\shortstack[l]{\texttt{Glass}\\$(214,9,6)$}} 
        & Succ.    & \textbf{10}/10 & \textbf{10}/10 & \textbf{10}/10 & \textbf{10}/10 & \textbf{10}/10 & \textbf{10}/10 \\
        & Iter.    & $20.2\pm4.9$ & $22.7\pm8.8$ & $\textbf{15.0}\pm0.0$ & $\textbf{15.0}\pm0.0$ & $\textbf{15.0}\pm0.0$ & $\textbf{15.0}\pm0.0$ \\
        & Obj.     & $\mathbf{1.949\mathrm{e}{+00}}\pm1.101\mathrm{e}{-03}$ & $\mathbf{1.949\mathrm{e}{+00}}\pm1.343\mathrm{e}{-03}$ & $1.954\mathrm{e}{+00}\pm1.256\mathrm{e}{-14}$ & $1.954\mathrm{e}{+00}\pm1.256\mathrm{e}{-14}$ & $1.954\mathrm{e}{+00}\pm1.256\mathrm{e}{-14}$ & $1.954\mathrm{e}{+00}\pm1.256\mathrm{e}{-14}$ \\
        & Subprob. & $20.2\pm4.9$ & $25.0\pm10.2$ & $\textbf{15.0}\pm0.0$ & $\textbf{15.0}\pm0.0$ & $\textbf{15.0}\pm0.0$ & $30.0\pm0.0$ \\
        & Time     & $\textbf{0.04}\pm0.01$ & $\textbf{0.04}\pm0.02$ & $\textbf{0.04}\pm0.00$ & $\textbf{0.04}\pm0.00$ & $\textbf{0.04}\pm0.00$ & $0.06\pm0.00$ \\
        \cmidrule(lr){1-8}
        
        \multirow{5}{*}{\shortstack[l]{\texttt{Yeast}\\$(1484,8,10)$}} 
        & Succ.    & \textbf{10}/10 & \textbf{10}/10 & 7/10 & \textbf{10}/10 & 5/10 & 7/10 \\
        & Iter.    & $17.1\pm5.6$ & $13.9\pm4.3$ & $21.3\pm24.7$ & $\textbf{5.2}\pm0.4$ & $30.7\pm26.5$ & $20.9\pm24.5$ \\
        & Obj.     & $\mathbf{3.014\mathrm{e}{-01}}\pm3.671\mathrm{e}{-04}$ & $\mathbf{3.014\mathrm{e}{-01}}\pm4.560\mathrm{e}{-04}$ & $3.056\mathrm{e}{-01}\pm4.044\mathrm{e}{-04}$ & $3.055\mathrm{e}{-01}\pm4.308\mathrm{e}{-04}$ & $3.056\mathrm{e}{-01}\pm4.128\mathrm{e}{-04}$ & $3.056\mathrm{e}{-01}\pm4.135\mathrm{e}{-04}$ \\
        & Subprob. & $17.1\pm5.6$ & $\textbf{15.6}\pm4.8$ & $21.3\pm24.7$ & $520.0\pm42.2$ & $30.7\pm26.5$ & $41.8\pm48.9$ \\
        & Time     & $8.99\pm4.21$ & $\textbf{8.29}\pm2.87$ & $95.98\pm140.95$ & $34.97\pm2.37$ & $153.10\pm154.93$ & $95.72\pm141.05$ \\
        \bottomrule
    \end{tabular}%
    }
    \vspace{0.5em}
    \end{minipage}%
    }
\end{table}

\subsubsection{Numerical results}
Table~\ref{tab:kmedians} presents {the number of successful trials, the number of iterations, the final objective value, the number of subproblems, and CPU time} of the compared algorithms for the $K$-medians model across four UCI datasets.
{On the relatively simple Iris, Wine, and Glass datasets, all perturbed DCA-type methods exhibit competitive performance. On the Yeast dataset, which contains many ties, the perturbed DCA and hybrid perturbed DCA maintain a 100\% success rate while hybrid perturbed DCA requires the least computation time and has the lowest objective value (tied with perturbed DCA). Revised DCA also attains a 100\% success rate, but it solves substantially more subproblems and therefore takes more time. DCA performs well when it succeeds, yet it can terminate without satisfying the d-stationarity criterion. Finally, Revised DCA-Rand and Hybrid Random Index-DCA depend on the randomly selected indices from \(\mathcal M_\epsilon\). When $\mc M_\epsilon$ (or \(\mathcal M\)) is large, identifying a useful direction can require many sampled attempts, so the time limit may be reached first.} These results illustrate the {robustness of the proposed perturbed DCA-type methods} for computing high-quality d-stationary solutions in the $K$-medians clustering problem.

\section{Conclusion}\label{sec:conclusion}
In this paper, we proposed a simple yet effective algorithm called perturbed DCA for computing d-stationary points of nonsmooth DC programs in the form of \eqref{eq:dc_general}. Benefiting from the tiny random perturbations at each iteration, the active gradient set of the concave component is a singleton almost surely. This allows perturbed DCA to solve a single strongly convex subproblem for updating the current solution in each iteration with comparable per-iteration computational cost to DCA. {We further introduced a hybrid variant of perturbed DCA that independently samples the perturbation radius and direction with a safeguard using a proximal DCA step. Under some practical assumptions}, every accumulation point of the sequence generated by the (hybrid) perturbed DCA is a d-stationary point of the nonsmooth DC program \eqref{eq:dc_general} almost surely. The efficiency of the proposed methods for computing d-stationary points has been demonstrated numerically. Whether the entire sequence generated by perturbed DCA or hybrid perturbed DCA converges almost surely to a d-stationary point under practical assumptions remains an open question, which we leave for future work. 

\section*{Acknowledgments}
The authors would like to thank Professor Jong-Shi Pang at the University of Southern California for his helpful discussions and suggestions on a preliminary draft of this paper.

\bibliography{reference}{}
\bibliographystyle{siam}
\end{document}